\documentclass[11pt,twoside]{article}
\usepackage{mathrsfs}
\usepackage{amssymb}
\usepackage{amsmath}
\usepackage[all]{xy}
\usepackage{amsthm}

\usepackage{url}
\numberwithin{equation}{section}

\newtheorem{theorem}{Theorem}[section]
\newtheorem{proposition}[theorem]{Proposition}
\newtheorem{lemma}[theorem]{Lemma}

\newtheorem{corollary}[theorem]{Corollary}

\newtheorem{theorem*}{Theorem}

\theoremstyle{definition}
\newtheorem{definition}[theorem]{Definition}

\theoremstyle{remark}
\newtheorem{remark}[theorem]{Remark}

\usepackage{mathrsfs}
\usepackage{url}
\usepackage[all]{xy}
\def\im{\mathop{\rm Im}\nolimits}
\def\ker{\mathop{\rm Ker}\nolimits}
\def\coker{\mathop{\rm Coker}\nolimits}

\def\Mod{\mathop{\rm Mod}\nolimits}
\def\wfd{\mathop{\rm wfd}\nolimits}
\def\wid{\mathop{\rm wid}\nolimits}

\def\id{\mathop{\rm id}\nolimits}
\def\pd{\mathop{\rm pd}\nolimits}

\def\inf{\mathop{\rm inf}\nolimits}

\def\Hom{\mathop{\rm Hom}\nolimits}
\def\Ext{\mathop{\rm Ext}\nolimits}
\def\Tor{{\rm Tor}}
\def\ra{\rightarrow}
\def\lra{\longrightarrow}

\def\lim{\mathop{\underrightarrow{\rm lim}}\nolimits}

\title{ \bf Foxby equivalence relative to $C$-weak injective and $C$-weak flat modules
}
\vspace{0.2cm}

\author{Zenghui Gao\thanks{E-mail address: gaozenghui@cuit.edu.cn}, \ Tiwei Zhao\thanks{E-mail address:  tiweizhao@hotmail.com} \\
{\it \scriptsize  1. College of Applied Mathematics, Chengdu University of Information Technology, Chengdu 610225, P.R. China}\\
{\it \scriptsize  2. Department of Mathematics, Nanjing University, Nanjing 210093,  P.R. China}}
\date{ }
\begin{document}

\baselineskip=16pt
\maketitle

\begin{abstract}
Let $S$ and $R$ be rings and $_SC_R$ a (faithfully) semidualizing bimodule.
We introduce and study $C$-weak flat and $C$-weak injective modules as a generalization of  $C$-flat and $C$-injective modules
(J. Math. Kyoto Univ. {\bf 47}(2007), 781--808) respectively, and use them to provide additional information concerning
the important Foxby equivalence between the subclasses of the Auslander class $\mathcal{A}_C(R)$ and that of the
Bass class $\mathcal{B}_C(S)$. Then we study the stability of Auslander and Bass classes, which enables us to
give some alternative characterizations of the modules in  $\mathcal{A}_C(R)$ and  $\mathcal{B}_C(S)$.
Finally we consider an open question which is closely relative to the main results (Proc. Edinb. Math. Soc. {\bf 48}(2005), 75--90),
and discuss the relationship between the Bass class $\mathcal{B}_C(S)$ and the class of Gorenstein injective modules.
\end{abstract}

\pagestyle{myheadings}
\markboth{\rightline {\scriptsize  Z. Gao, T. Zhao}}
         {\leftline{\scriptsize  Foxby equivalence relative to $C$-weak injective and $C$-weak flat modules}}

\section*{Introduction}

Over a commutative Noetherian ring $R$, a finitely generated $R$-module
$C$ is \textit{semidualizing} if the natural homothety morphism $R\ra \Hom_R(C,C)$ is an isomorphism
and $\Ext^i_R(C,C)=0$ for all $i\geq 1$. Semidualizing modules (under different names) were independently studied by Foxby, Golod and Vasconcelos (see \cite{Fo73,Go84,Va74}).
In \cite{Ch01}, Christensen extended this notion to semidualizing complexes.
Araya, Takahashi and Yoshino \cite{ATY05} extended the notion of semidualizing modules to a pair of non-commutative, but Noetherian rings.
Furthermore, Holm and White in \cite{HW07} generalized the notion of
a semidualizing module to general associative rings, and defined and
studied Auslander and Bass classes with respect to a semidualizing
bimodule $C$. They obtained some beautiful characterizations of the
modules in the Auslander and Bass classes in terms of $C$-injective,
$C$-projective and $C$-flat modules and showed Foxby equivalence
between the subclasses of the Auslander class and that of the Bass
class. In particular, it was proven in \cite[Lemma 4.1]{HW07} that
the Auslander class $\mathcal{A}_C(R)$ contains all flat left
$R$-modules and the Bass class $\mathcal{B}_C(S)$ contains all
injective left $S$-modules. Recently, Bennis et al. continued a
study of homological notions relative to an extension of a
semidualizing module (see \cite{BGO16a,BGO16b,BGO16c}).

More recently, Gao and his coauthors introduced and studied in \cite{GH15a,GW15} a
generalization of injective and flat modules, named weak injective
and weak flat modules respectively, and generalized many homological
results from coherent rings to arbitrary rings. In this process
finitely presented modules are replaced by super finitely presented
modules. In \cite{BGH14}, Bravo, Gillespie and Hovey described
how Gorenstein homological algebra should work for general rings,
and the weak injective and weak flat modules were also called
$FP_\infty$-injective (or absolutely clean) and level modules
respectively. Following the above philosophy, the following question
naturally arises in this situation: \vspace{0.1cm}

{\bf Question 1.} Is it true that the Auslander class $\mathcal{A}_C(R)$
contains all weak flat left $R$-modules and the Bass class $\mathcal{B}_C(S)$
contains all weak injective left $S$-modules? \vspace{0.1cm}

In \cite{EJL05}, Enochs, Jenda and L\'{o}pez-Ramos proved that if
$R$ and $S$ are right and left Noetherian rings respectively
admitting a dualizing bimodule (see \cite[Definition 3.1]{EJL05}),
then all Gorenstein projective left $R$-modules are in
$\mathcal{A}(R)$ (\cite[Proposition 3.9]{EJL05}) and all Gorenstein
injective left $S$-modules are in $\mathcal{B}(S)$
(\cite[Proposition 3.8]{EJL05}). Moreover, if every flat left
$R$-module has finite projective dimension, then a left $S$-module
$N\in \mathcal{B}(S)$ if and only if $N$ has finite Gorenstein
injective dimension by \cite[Lemma 3.15 and Proposition
3.13]{EJL05}; dually, we can deduce that a left $R$-module $M\in
\mathcal{A}(R)$ if and only if $M$ has finite Gorenstein projective
dimension. In view of the relationship between the Auslander class
(resp. Bass class) and the class of Gorenstein projective (resp.
Gorenstein injective) modules in \cite{EJL05}, it is natural to ask
the following question:\vspace{0.1cm}

{\bf Question 2.}  Is there an appropriate semidualizing bimodule $C$
such that the Auslander class $\mathcal{A}_C(R)$ contains all Gorenstein
projective $R$-modules and the Bass class $\mathcal{B}_C(S)$ contains all
Gorenstein injective $S$-modules?\vspace{0.1cm}

The aim of this paper is to study these two questions, and we will define and
investigate $C$-weak injective and $C$-weak flat modules with respect to a semidualizing bimodule $C$. Suppose that $C$ is a faithfully semidualizing bimodule. We provide additional information concerning the important Foxby equivalence between the subclasses of Auslander class $\mathcal{A}_C(R)$ and that of the Bass class $\mathcal{B}_C(S)$. In addition, we study the stability of the Auslander and Bass classes, and some new characterizations of the modules in the Auslander and Bass classes are given. We will answer Question 1 in Theorem \ref{2.2}, and give a partial
answer to Question 2 at the end of the paper, that is, it is shown that if ${}_SC_R$ is a faithfully semidualizing bimodule with finite $S$-projective dimension,
then every Gorenstein injective left $S$-module is in $\mathcal{B}_C(S)$. This paper is organized as follows.

In Section 1, we give some terminology and some preliminary results.

In Section 2, we introduce the notions of $C$-weak injective and $C$-weak flat
modules with respect to a semidualizing bimodule $C$, and prove that
the Auslander class $\mathcal{A}_C(R)$ contains all weak flat left
$R$-modules and the Bass class $\mathcal{B}_C(S)$ contains all weak
injective left $S$-modules. We show that $C$-weak injective and
$C$-weak flat modules possess many nice properties analogous to that of $C$-injective
and $C$-flat (or $C$-projective) modules as in \cite{HW07}. For
example, we prove that the classes $\mathcal{WF}_C(S)$ and
$\mathcal{WI}_C(R)$, consisting of all $C$-weak flat left
$S$-modules and all $C$-weak injective left $R$-modules
respectively, are closed under direct summands, direct products,
direct sums and direct limits. Also, both of them are closed under
pure submodules and pure quotients. As a consequence, we obtain that
the classes $\mathcal{WF}_C(S)$ and $\mathcal{WI}_C(R)$ are covering
and preenveloping.

In Section 3, we investigate Foxby equivalence relative to $C$-weak injective and $C$-weak flat modules. The following is Theorem \ref{3.4}.
Here $\mathcal{WF}(R)_{\leq n}$ and $\mathcal{WI}(S)_{\leq n}$ stand for the class of left $R$-modules of weak flat dimension at most $n$ and the class of left $S$-modules of weak injective dimension at most $n$, respectively; and $\mathcal{WF}_C(S)_{\leq n}$ and $\mathcal{WI}_C(R)_{\leq n}$ denote the class of left $S$-modules of $C$-weak flat dimension at most $n$ and the class of left $R$-modules of $C$-weak injective dimension at most $n$, respectively.\vspace{0.1cm}

{\bf Theorem A.} {\rm (Foxby Equivalence)} {\it There are equivalences of categories
$$\xymatrix@R=20pt@C=60pt{
\mathcal{WF}(R) \ar@<+4pt>[r]^{C\otimes_R-}\ar@{^{(}->}[d] & \mathcal{WF}_C(S) \ar@<+5pt>[l]^{\Hom_S(C,-)}_{\sim}\ar@{^{(}->}[d]\\
\mathcal{WF}(R)_{\leq n} \ar@<+4pt>[r]^{C\otimes_R-}\ar@{^{(}->}[d] & \mathcal{WF}_C(S)_{\leq n} \ar@<+5pt>[l]^{\Hom_S(C,-)}_{\sim}\ar@{^{(}->}[d]\\
\mathcal{A}_C(R) \ar@<+4pt>[r]^{C\otimes_R-} & \mathcal{B}_C(S) \ar@<+5pt>[l]^{\Hom_S(C,-)}_{\sim}\\
\mathcal{WI}_C(R)_{\leq n} \ar@<+4pt>[r]^{C\otimes_R-}\ar@{^{(}->}[u] & \mathcal{WI}(S)_{\leq n} \ar@<+5pt>[l]^{\Hom_S(C,-)}_{\sim}\ar@{^{(}->}[u]\\
\mathcal{WI}_C(R)           \ar@<+4pt>[r]^{C\otimes_R-}\ar@{^{(}->}[u] & \mathcal{WI}(S).          \ar@<+5pt>[l]^{\Hom_S(C,-)}_{\sim}\ar@{^{(}->}[u]\\}
$$}\vspace{0.1cm}

In Section 4, we characterize the stability of Auslander class $\mathcal{A}_C(R)$ and the Bass class $\mathcal{B}_C(S)$, and then give some applications of them. Motivated by \cite[Theorem A]{SSW08}, we show that an iteration of the procedure used to
describe the Auslander class yields exactly the Auslander class, which
generalizes \cite[Theorem 2]{HW07}.
That is, we set $[\mathcal{A}_C(R)]^1=\mathcal{A}_C(R)$,
and inductively set $[\mathcal{A}_C(R)]^{n+1}=\{M\in \Mod R\mid$
there exists a $C\otimes_R-$ exact exact sequence $\cdots\ra W_1\ra
W_0\ra W^0\ra W^1\ra\cdots$ in $\Mod R$ with all $W_i$ and $W^i$ in
$[\mathcal{A}_C(R)]^{n}$ such that $M\cong\coker(W_1\ra W_0)\}$ for
any $n\geq 1$. Similarly, we inductively set
$[\mathcal{B}_C(S)]^{n}$. The following are Theorems \ref{4.7} and \ref{4.8}.

\vspace{0.1cm}

{\bf Theorem B.} {\it $[\mathcal{A}_C(R)]^{n}=\mathcal{A}_C(R)$ and $[\mathcal{B}_C(S)]^{n}=\mathcal{B}_C(S)$ for any $n\geq 1$. }

\section{Preliminaries}

In this section, we give some terminology and some
preliminary results needed in the sequel. For more details the
reader can consult \cite{EJ00,GW15,GT12,HJ06,HW07,TW10}.

\vspace{0.2cm}

{\bf 1.1} Throughout this paper, $R$ and $S$ are fixed associative rings with unites, and all modules are unitary.
We use $\Mod R$ or $\Mod S$ to stand for the class of left $R$- or $S$-modules.
Right $R$- or $S$-modules are identified with left modules over the opposite rings $R^{op}$ or $S^{op}$.
The notation ${}_SM_R$ is used to indicate that $M$ is an $(S,R)$-bimodule,
and the structures are compatible in the sense that $s(xr)=(sx)r$ for all $s\in S, r\in R, x\in M$.
For a left or right $R$-module $M$, $M^+=\Hom_\mathbb{Z}(M,\mathbb{Q}/\mathbb{Z})$.

\vspace{0.2cm}

{\bf 1.2.} A \textit{degreewise finite projective resolution} of a left $R$-module $M$ is a projective resolution of $M$:\ \ $\cdots\ra P_n\ra \cdots\ra P_1\ra P_0\ra M\ra 0$ in $\Mod R$ with each $P_i$ finitely generated projective.
Note that a left $R$-module admitting a degreewise finite projective resolution is also called $FP_{\infty}$ in \cite{BGH14,Br82,HM09}, \textit{infinitely presented} in \cite{BM07}, \textit{strongly finitely presented} in \cite{GT12}, and \textit{super finitely presented} in \cite{GW15}. Also, it is shown that this class of modules plays a crucial role in the process of generalizing many homological results from coherent rings to arbitrary rings (see \cite{BGH14, GH15a,GW15}).

\vspace{0.2cm}

{\bf 1.3.}  An $(S,R)$-bimodule $C={}_SC_R$ is \textit{semidualizing} if

(a1) ${}_SC$ admits a degreewise finite projective resolution in $\Mod S$.

(a2) $C_R$ admits a degreewise finite projective resolution in $\Mod R^{op}$.

(b1) The homothety map ${}_SS_S\buildrel{{}_S\gamma}\over\lra\Hom_{R^{op}}(C,C)$ is an isomorphism.

(b2) The homothety map ${}_RR_R\buildrel{\gamma_R}\over\lra\Hom_{S}(C,C)$ is an isomorphism.

(c1) $\Ext_S^i(C,C)=0$ for all $i\geq 1$.

(c2) $\Ext_{R^{op}}^i(C,C)=0$ for all $i\geq 1$.

A semidualizing bimodule ${}_SC_R$ is \textit{faithfully semidualizing} if
 it satisfies the following conditions for all modules ${}_SN$ and $M_R$:

(1) If $\Hom_S(C,N)=0$, then $N=0$.

(2) If $\Hom_{R^{op}}(C,M)=0$, then $M=0$.\vspace{0.2cm}

By definition, it follows that every semidualizing module is super finitely presented as a left $S$-module or a right $R$-module. It was shown in \cite[Proposition 3.1]{HW07} that if $R=S$ is commutative, then every semidualizing $R$-module is faithfully semidualizing. Also in \cite{HW07} many examples of faithfully semidualizing bimodules were provided over a wide class of non-commutative rings.

\vspace{0.2cm}

{\bf 1.4.}  The \textit{Auslander class} $\mathcal{A}_C(R)$ with respect to $C$ consists of all modules $M$ in $\Mod R$ satisfying:

(A1) $\Tor_i^R(C,M)=0$ for all $i\geq 1$.

(A2) $\Ext_S^i(C,C\otimes_RM)=0$ for all $i\geq 1$.

(A3) The natural evaluation homomorphism $\mu_{_M}: M\longrightarrow\Hom_S(C,C\otimes_RM)$ is an isomorphism (of left $R$-modules).

The \textit{Bass class} $\mathcal{B}_C(S)$ with respect to $C$ consists of all modules $N\in\Mod S$ satisfying:

(B1) $\Ext_S^i(C,N)=0$ for all $i\geq 1$.

(B2) $\Tor^R_i(C,\Hom_S(C,N))=0$ for all $i\geq 1$.

(B3) The natural evaluation homomorphism $\nu_{_N}: C\otimes_R\Hom_S(C,N)\longrightarrow N$ is an isomorphism (of left $S$-modules).

It is an important property of Auslander and Bass classes that they are equivalent under the pair of functors (\cite[Proposition 4.1]{HW07}):
$$\xymatrix@C=80pt{ \mathcal{A}_C(R) \ar@<+4pt>[r]^{C\otimes_R-} & \mathcal{B}_C(S) \ar@<+5pt>[l]^{\Hom_S(C,-)}_{\sim}. }
$$

{\bf 1.5.} Let $\mathcal{F}$ be a subcategory of $\Mod R$. The homomorphism
$f:F\ra M$ in $\Mod R$ with $F\in \mathcal{F}$ is an \textit{$\mathcal{F}$-precover} of $M$ if for any
homomorphism $g:F_0\ra M$ in $\Mod R$ with $F_0\in \mathcal{F}$, there exists a homomorphism
$h:F_0\ra F$ such that the following diagram commutes:
$$\xymatrix{ & F_0 \ar[d]^{g} \ar@{-->}[ld]_{h}\\
F \ar[r]^{f} & M.}$$
The homomorphism $f:F\ra M$ is \textit{right minimal} if an endomorphism $h:F\ra F$ is an automorphism whenever $f=fh$.
An $\mathcal{F}$-precover $f:F\ra M$ is an \textit{$\mathcal{F}$-cover} if $f$ is right minimal. We say that $\mathcal{F}$ is
{\it (pre)covering} if every module in $\Mod R$ admits an $\mathcal{F}$-(pre)cover.

Dually, the notions of an {\it $\mathcal{F}$-preenvelope}, a {\it left minimal
homomorphism}, an {\it $\mathcal{F}$-envelope} and a {\it (pre)enveloping subcategory} are defined.

\vspace{0.2cm}

{\bf 1.6.} We say that a
sequence $\mathbf{X}=\ \ \cdots\ra X_1\ra X_0\ra X_{-1}\ra \cdots$ in $\Mod R$ (resp. in $\Mod S^{op}$) is $C\otimes_R-$ (resp. $-\otimes_S C$) exact if
the complex $C\otimes_R\mathbf{X}$ (resp. $\mathbf{X}\otimes_S C$) is
exact; and a sequence $\mathbf{X}$  in $\Mod S$ is $\Hom_S(C,-)$ (resp.
$\Hom_S(-,C)$) exact if the complex $\Hom_S(C,\mathbf{X})$ (resp.
$\Hom_S(\mathbf{X},C)$) is exact.

We denote by $\mathcal{X}$ a fixed class of left $R$-modules.
An \textit{$\mathcal{X}$-resolution} of a left $R$-module $M$ is an
exact sequence $\mathbf{X}=\ \ \cdots\ra X_1\ra X_0\ra M\ra 0$ in $\Mod R$ with
$X_i\in \mathcal{X}$ for all $i\geq 0$. An
\textit{$\mathcal{X}$-coresolution} of a left $R$-module $M$ is an exact
sequence $\mathbf{X}=\ \ 0 \ra M\ra X^0\ra X^1\ra\cdots$ in $\Mod R$ with
$X^i\in \mathcal{X}$ for all $i\geq 0$.

If the class $\mathcal{X}$ is precovering, then for any left $R$-module $M$,
there exists an \textit{augmented proper $\mathcal{X}$-resolution} of $M$, that is,
a complex $$\mathbf{X}=\ \ \cdots\overset{\partial_{2}^\mathbf{X}}\lra X_1\overset{\partial_{1}^\mathbf{X}}\lra X_0\lra M\lra 0$$
 in $\Mod R$ such that it is $\Hom_R(X,-)$ exact for each $X\in \mathcal{X}$. The truncated complex
$$\mathbf{X}_M=\ \ \cdots\overset{\partial_{3}^\mathbf{X}}\lra X_2\overset{\partial_{2}^\mathbf{X}}\lra X_1\overset{\partial_{1}^\mathbf{X}}\lra X_0\lra 0$$
is a \textit{proper $\mathcal{X}$-resolution} of $M$.

In general, an augmented proper $\mathcal{X}$-resolution $\mathbf{X}$ need not be exact. However, the complex $\mathbf{X}$ is exact if $\mathcal{X}$ contains all projective left $R$-modules. Dually, the augmented coproper $\mathcal{X}$-coresolutions are defined, and they must be exact if the class $\mathcal{X}$ contains all injective left $R$-modules.

\vspace{0.2cm}

{\bf 1.7.} A module in $\Mod S$ is \textit{$C$-flat} (resp. \textit{$C$-projective}) if it has the form $C\otimes_RF$ for some flat
(resp. projective) module $F\in\Mod R$. A module in $\Mod R$ is \textit{$C$-injective}
if it has the form $\Hom_S(C,I)$ for some injective module $I\in\Mod S$. We set
\[\begin{array}{ll}\vspace{0.1cm}
&\mathcal{F}_C(S)=\{C\otimes_RF\mid F \mbox{ is a flat left $R$-module}\},\\\vspace{0.1cm}
&\mathcal{P}_C(S)=\{C\otimes_RF\mid P \mbox{ is a projective left $R$-module}\},\\\vspace{0.1cm}
&\mathcal{I}_C(R)=\{\Hom_S(C,I)\mid I \mbox{ is an injective left $S$-module}\}.\\
 \end{array}\]

 Over a commutative ring $R$, the notions of $C$-projective and $C$-injective dimensions of an $R$-module were introduced in \cite{TW10}.
That is, the \textit{$\mathcal{P}_C$-projective dimension} of an $R$-module $M$ is
$$\mathcal{P}_C\mbox{-}\pd(M)=\mbox{inf}\{\mbox{sup}\{n\mid X_n\neq 0\}\mid \mbox{$X$ is a proper $\mathcal{P}_C$-projective resolution of $M$}\}.$$
The \textit{$\mathcal{I}_C$-injective dimension}, denoted by $\mathcal{I}_C$-$\id(-)$, can be defined dually.
It was also proven in \cite[Corollary 2.10]{TW10} that

 (a) $\mathcal{P}_C\mbox{-}\pd(M)\leq n$ if and only if there is an exact sequence
 $$0\ra C\otimes_RP_n\ra\cdots\ra C\otimes_RP_1\ra C\otimes_RP_0\ra M\ra 0$$ with each $P_i$ a projective  $R$-module.

 (b) $\mathcal{I}_C\mbox{-}\id(M)\leq n$ if and only if there is an exact sequence
 $$0\ra M\ra \Hom_R(C,I^0)\ra\Hom_R(C,I^1)\ra\cdots\ra\Hom_R(C,I^n)\ra 0$$ with each $I^i$ an injective $R$-module.

\vspace{0.2cm}

{\bf 1.8.} A module $M$ in $\Mod R$ (resp. $N$ in $\Mod R^{op}$)
is \textit{weak injective} (resp. \textit{weak flat}) if $\Ext_R^1(F,M)=0$
(resp. $\Tor^R_1(N,F)=0$) for any super finitely presented left $R$-module $F$.
We use $\mathcal{WI}(R)$ (resp. $\mathcal{WF}(R^{op})$)
to denote the full subcategory of $\Mod R$ (resp. $\Mod R^{op}$)
consisting of weak injective modules (resp. weak flat modules).

The weak injective
dimension of a module $M$ in $\Mod R$, denoted by $\wid_R(M)$, is defined as
\begin{align*}
  \wid_R(M)= \inf\left\{ \right.&n\mid\Ext_R^{n+1}(F,M)=0\mbox{ for any super} \\
   &  \left.\mbox{finitely presented left $R$-module $F$}\right\}.
\end{align*}
$
$ If no such $n$ exists, we set
$\wid_R(M)=\infty$.
Dually, the weak flat dimension $\wid_R(-)$ of a module is defined.

\vspace{0.2cm}
 {\bf 1.9.} Given a short exact sequence $0\ra A\ra B\ra C\ra 0$ in $\Mod R$ where $A$
is a submodule of $B$ and $C$ is the corresponding quotient module. The sequence is said to be \textit{pure exact} if $\Hom_R(P,B)\ra \Hom_R(P,C)\ra 0$ is exact for any
finitely presented module $P$ in $\Mod R$, or equivalently, if $0\ra M\otimes_R A\ra M\otimes_R B$
is exact for any module $M$ in $\Mod R^{op}$.
In this case, $A$ and $C$ are called a \textit{pure submodule} and a \textit{pure quotient} of $B$ respectively.


\section{$C$-weak injective and $C$-weak flat modules}

In this section, we give a treatment of $C$-weak injective and $C$-weak flat modules with respect to a (faithful) semidualizing bimodule $C$.

\begin{definition}\label{2.1}  A module in $\Mod S$ is called \textit{$C$-weak flat}  if
it has the form $C\otimes_RF$ for some weak flat module $F\in \Mod R$.
A module in $\Mod R$ is called \textit{$C$-weak injective} if it has the form $\Hom_S(C,I)$ for some weak injecitve module $I\in \Mod S$. We set
\[\begin{array}{ll}\vspace{0.1cm}
&\mathcal{WF}_C(S)=\{C\otimes_RF\mid F \mbox{ is a weak flat left $R$-module}\},\\
&\mathcal{WI}_C(R)=\{\Hom_S(C,I)\mid I \mbox{ is a weak injective left $S$-module}\}.\\
 \end{array}\]

The \textit{$C$-weak flat dimension} of a module $M\in \Mod S$ is defined that
 $C\mbox{-}\wfd_S(M)\leq n$ if and only if there is an exact sequence
 $$0\ra C\otimes_RF_n\ra\cdots\ra C\otimes_RF_1\ra C\otimes_RF_0\ra M\ra 0$$ in $\Mod S$ with each $F_i$ in $\mathcal{WF}(R)$. If no such $n$ exists, set
 $C\mbox{-}\wfd_S(M)=\infty$.

The \textit{$C$-weak injective dimension} of a module $M\in \Mod R$ is defined that  $C\mbox{-}\wid_R(M)\leq n$ if and only if there is an exact sequence
 $$0\ra M\ra \Hom_S(C,I^0)\ra\Hom_S(C,I^1)\ra\cdots\ra\Hom_S(C,I^n)\ra 0$$ in $\Mod R$ with each $I^i$ in $\mathcal{WI}(S)$. If no such $n$ exists, set
$C\mbox{-}\wid_R(M)=\infty$. \end{definition}

The following theorem gives an affirmative answer to Question 1.

\begin{theorem}\label{2.2}  {The Auslander class $\mathcal{A}_C(R)$ contains all weak flat modules in $\Mod R$,
and the Bass class $\mathcal{B}_C(S)$ contains all weak injective modules in $\Mod S$.}\end{theorem}

\begin{proof} We only show that $\mathcal{A}_C(R)$ contains all weak flat modules in $\Mod R$, the other proof is dual.
Let $M$ be a weak flat module in $\Mod R$. Then $\Tor_i^R(C,M)=0$ for all $i\geq 1$ by \cite[Proposition 3.1]{GW15}.
Since ${}_SC_R$ is semidualizing, it follows that $\Ext_S^i(C,C)=0$ for all $i\geq 1$,
and there exists a $\Hom_S(-,C)$-exact exact sequence
$$\cdots\lra P_n\lra\cdots\lra P_1\lra P_0\lra C\lra 0\eqno{(2.1)}$$
in $\Mod S$ with all $P_i$ finite generated projective. This gives rise to the exactness of
$$0\ra\Hom_S(C,C)\ra \Hom_S(P_0,C)\ra  \cdots\ra \Hom_S(P_n,C)\ra\cdots.$$
One easily checks that each $\Hom_S(P_i,C)$ is super finitely presented by applying $\Hom_S(P_i,-)$ to the sequence (2.1).
Now that $\Hom_S(C,C)\cong R$ since ${}_SC_R$ is semidualizing,
it follows that all $\im(\Hom_S(P_i,C)\ra \Hom_S(P_{i+1},C))$ are super finitely presented by \cite[Lemma 2.3]{HM09}.
Because $M$ is a weak flat left $R$-module, one gets the following exact sequence
$$0\ra\Hom_S(C,C)\otimes_RM\ra \Hom_S(P_0,C)\otimes_RM\ra\Hom_S(P_1,C)\otimes_RM\ra\cdots.$$
By the tensor evaluation morphism (cf.\cite[1.10]{HW07}), we have the following isomorphisms:
$$\omega_{P_iCM}:\Hom_S(P_i,C)\otimes_RM\buildrel{\cong}\over\lra\Hom_S(P_i,C\otimes_RM) \ \ \mbox{for all } i\geq 0$$
since $P_i$ is finitely generated projective. Thus we obtain the following commutative diagram with exact rows:
$$\xymatrix@C=10pt@R=15pt{0 \ar[r] & \Hom_S(C,C)\otimes_RM \ar[d]_{\omega_{CCM}} \ar[r] & \Hom_S(P_0,C)\otimes_RM
\ar[d]_{\omega_{P_0CM}}^{\cong}\ar[r] & \Hom_S(P_1,C)\otimes_RM \ar[d]_{\omega_{P_1CM}}^{\cong} \\
0 \ar[r] & \Hom_S(C,C\otimes_RM) \ar[r] & \Hom_S(P_0,C\otimes_RM)  \ar[r] &\Hom_S(P_1,C\otimes_RM).}$$
Hence $\omega_{CCM}:\Hom_S(C,C)\otimes_RM\ra\Hom_S(C,C\otimes_RM)$ is an isomorphism by the five lemma,
which implies that $\mu_M:M\ra\Hom_S(C,C\otimes_RM)$ is an isomorphism.
Now consider the following commutative diagram with the upper row exact:
$$\xymatrix@C=10pt@R=15pt{0 \ar[r] & \Hom_S(C,C)\otimes_RM \ar[d]^{\cong} \ar[r] & \Hom_S(P_0,C)\otimes_RM
\ar[d]^{\cong}\ar[r] & \Hom_S(P_1,C)\otimes_RM \ar[d]^{\cong} \ar[r] & \cdots\\
 0 \ar[r] & \Hom_S(C,C\otimes_RM) \ar[r] & \Hom_S(P_0,C\otimes_RM)  \ar[r] &\Hom_S(P_1,C\otimes_RM)\ar[r] & \cdots.}$$
 Then we obtain the exactness of $$0\ra\Hom_S(C,C\otimes_RM)\ra \Hom_S(P_0,C\otimes_RM)\ra\Hom_S(P_1,C\otimes_RM)\ra\cdots.$$
It follows that $\Ext_S^i(C,C\otimes_RM)=0$ for all $i\geq 1$, and hence $M\in \mathcal{A}_C(R)$.
 \end{proof}

In what follows, $C={}_SC_R$ always stands for a faithfully semidualizing bimodule. By Theorem \ref{2.2} and \cite[Theorem 6.3]{HW07}, we immediately get the following result.

\begin{corollary}\label{2.3} {The Auslander class $\mathcal{A}_C(R)$ contains all modules in $\Mod R$ of finite weak flat dimension,
the Bass class $\mathcal{B}_C(S)$ contains all modules in $\Mod S$ of finite weak injective dimension. }\end{corollary}

The following result plays a fundamental role in this paper.

\begin{proposition}\label{2.4}  {The following statements hold for modules ${}_SV$ and ${}_RU$:

$\mathrm{(1)}$ $V\in\mathcal{WF}_C(S)$ if and only if $V\in\mathcal{B}_C(S)$ and $\Hom_S(C,V)$ is weak flat over $R$.

$\mathrm{(2)}$ $U\in\mathcal{WI}_C(R)$ if and only if $U\in\mathcal{A}_C(R)$ and $C\otimes_RU$ is weak injective over $S$. }
\end{proposition}

\begin{proof} We only prove (1), and (2) is dual.

``Only if\," part. Let $V\in\mathcal{WF}_C(S)$. Then $V=C\otimes_RF$ for some weak flat left $R$-module $F$.
Note that $F\in \mathcal{A}_C(R)$ by Theorem \ref{2.2}, it follows that $V\in \mathcal{B}_C(S)$ by \cite[Proposition 4.1]{HW07}. It is clear that
$\Hom_S(C,C\otimes_RF)\cong F$, and so $\Hom_S(C,V)$ is a weak flat left $R$-module.

``If\," part. Suppose that $V\in\mathcal{B}_C(S)$ and $\Hom_S(C,V)$ is a weak flat left $R$-module. Then it is clear that
$V\cong C\otimes_R\Hom_S(C,V)$, and thus $V$ is a $C$-weak flat left $S$-module.\end{proof}

\begin{proposition}\label{2.5}  {The following statements hold.

$\mathrm{(1)}$  The class $\mathcal{WF}_C(S)$ is closed under extensions and kernels of epimorphisms.

$\mathrm{(2)}$  The class $\mathcal{WI}_C(R)$ is closed under extensions and cokernels of monomorphisms. }
\end{proposition}

\begin{proof} We only prove (1), and (2) is the dual of (1).

Let $$0\ra M'\ra M\ra M''\ra 0 \eqno{(2.2)}$$ be a short exact
sequence in $\Mod S$. If $M', M''\in \mathcal{WF}_C(S)$, then $M', M''\in \mathcal{B}_C(S)$
and $\Hom_S(C,M')$ and $\Hom_S(C,M'')$ are weak flat left $R$-modules by Proposition \ref{2.4}. It follows that $M\in\mathcal{B}_C(S)$ by \cite[Theorem 6.2]{HW07}.
Also, we have $\Ext^1_S(C,M')=0$.
Using $\Hom_S(C,-)$ to the sequence (2.2), we obtain the following exactness of
$$0\ra \Hom_S(C,M')\ra \Hom_S(C,M)\ra \Hom_S(C,M'')\ra 0$$ in $\Mod R.$ Then $\Hom_S(C,M)$ is weak flat
by \cite[Proposition 2.6(2)]{GH15a}. Hence $M\in\mathcal{WF}_C(S)$ by Proposition \ref{2.4}(1).

Assume that $M, M''\in \mathcal{WF}_C(S)$ in the sequence (2.2), then $M, M''\in \mathcal{B}_C(S)$,
and $\Hom_S(C,M)$ and $\Hom_S(C,M'')$ are weak flat left $R$-modules. By \cite[Theorem 6.3]{HW07}, we have $M'\in\mathcal{B}_C(S)$,
and so $\Ext^1_S(C,M')=0$.
Applying $\Hom_S(C,-)$ to the sequence (2.2), one gets the following exact sequence
$$0\ra \Hom_S(C,M')\ra \Hom_S(C,M)\ra \Hom_S(C,M'')\ra 0$$ in $\Mod R$.
Then, by \cite[Proposition 2.6(2)]{GH15a}, we have $\Hom_S(C,M')$ is weak flat.
It follows from Proposition \ref{2.4}(1) that $M'\in\mathcal{WF}_C(S)$, as desired.
 \end{proof}

The following result generalizes \cite[Theorem 2.10, and Remark 2.2(2)]{GW15}.

\begin{proposition}\label{2.6}  {The following statements hold for modules ${}_SV$ and ${}_RU$.

$\mathrm{(1)}$ $V\in \mathcal{WF}_C(S)$ if and only if $V^+\in \mathcal{WI}_C(S^{op})$.

$\mathrm{(2)}$ $U\in \mathcal{WI}_C(R)$ if and only if $U^+\in \mathcal{WF}_C(R^{op})$. }
\end{proposition}

\begin{proof} (1) ``Only if" part. Let $V\in \mathcal{WF}_C(S)$. Then by definition $V=C\otimes_RF$ for some weak flat left $R$-module $F$.
Since $F^+$ is weak injective in $\Mod R^{op}$ by \cite[Remark 2.2(2)]{GW15},
it follows that $V^+\cong\Hom_{R^{op}}(C,F^+)\in \mathcal{WI}_C(S^{op})$.

``If" part. Let $V^+\in \mathcal{WI}_C(S^{op})$. Then we have $V^+\in\mathcal{A}_C(S^{op})$ and $V^+\otimes_SC$ is weak injective over $R^{op}$ by Proposition \ref{2.4}(2).
It is not hard to check that $V\in\mathcal{B}_C(S)$ by a non-commutative version of \cite[Proposition 7.2(b) and Remark 4]{HW07}.
So we have the isomorphism: $V^+\otimes_SC\cong\Hom_S(C,V)^+$ by \cite[Lemma 2.16]{GT12} since $C$ is finitely presented.
It follows from \cite[Remark 2.2(2)]{GW15} that $\Hom_S(C,V)$ is weak flat. Thus $V\in \mathcal{WF}_C(S)$ by Proposition 2.4(1).

(2) The proof is similar to that of (1). \end{proof}

\begin{corollary}\label{2.7}  {The following statements hold.

$\mathrm{(1)}$ $V\in \mathcal{WF}_C(S)$ if and only if $V^{++}\in \mathcal{WF}_C(S)$.

$\mathrm{(2)}$ $U\in \mathcal{WI}_C(R)$ if and only if $U^{++}\in \mathcal{WI}_C(R)$.}
\end{corollary}

\begin{proof} The assertions follows immediately from Proposition \ref{2.6}. \end{proof}

\begin{proposition}\label{2.8}  {The classes $\mathcal{WF}_C(S)$ and $\mathcal{WI}_C(R)$ are closed under direct summands,
direct products, direct sums and direct limits.}\end{proposition}

\begin{proof} To prove $\mathcal{WF}_C(S)$ is closed under direct summands, we assume that
$$0\ra X_1\ra X_2\ra X_3\ra 0$$ is a split exact sequence in $\Mod S$ with $X_2\in \mathcal{WF}_C(S)$.
Then $X_2\in \mathcal{B}_C(S)$ and $\Hom_S(C,X_2)$ is weak flat by Proposition \ref{2.4}(1). It follows from \cite[Proposition 4.2]{HW07} that $X_1,X_3\in \mathcal{B}_C(S)$.
One easily checks that the sequence $$ 0\lra \Hom_S(C,X_1)\lra \Hom_S(C,X_2)\lra \Hom_S(C,X_3)\lra 0$$ in $\Mod R$ is split exact.
Then $\Hom_S(C,X_1)$ and $\Hom_S(C,X_3)$ are weak flat since the weak flat modules are closed under direct summands by \cite[Proposition 2.3]{GW15}.
It follows that $X_1,X_3\in \mathcal{WF}_C(S)$ by Proposition \ref{2.4}(1) again.

Let $\{F_\lambda\}_{\lambda\in \Lambda}$ be a family of $C$-weak flat modules in $\Mod S$.
Then $F_\lambda\in \mathcal{B}_C(S)$ and $\Hom_S(C,F_\lambda)$ is weak flat in $\Mod R$ for any $\lambda\in \Lambda$ by Proposition \ref{2.4}(1).
By \cite[Theorem 2.6]{Ro79}, we have the isomorphism
$$\prod_{\lambda\in \Lambda}\Hom_R(C,F_\lambda)\cong\Hom_R(C,\prod_{\lambda\in \Lambda}F_\lambda).$$
Note that $\prod_{\lambda\in \Lambda}\Hom_R(C,F_\lambda)$ is weak flat by \cite[Theorem 2.13]{GW15},
and $\prod_{\lambda\in \Lambda}F_\lambda\in\mathcal{B}_C(S)$ by \cite[Proposition 4.2]{HW07},
it follows that $\prod_{\lambda\in \Lambda}F_\lambda\in \mathcal{WF}_C(S)$ by Proposition \ref{2.4}(1).
Therefore $\mathcal{WF}_C(S)$ is closed under direct products.
Since $C$ is finitely presented and the weak flat modules are closed under direct sums,
similar to the arguments above, one can deduce that the class $\mathcal{WF}_C(S)$ is closed under direct sums.

Let $\{F_i\}_{i\in I}$ be a direct system of $C$-weak flat modules in $\Mod S$.
Then $F_i\in \mathcal{B}_C(S)$ for any $i\in I$ and $\{\Hom_S(C,F_i)\}_{i\in I}$
is a direct system of weak flat modules in $\Mod R$ by Proposition 2.4(1).
By \cite[Proposition 4.2]{HW07}, we get that $\lim F_i\in \mathcal{B}_C(S)$.
Note that the weak flat modules are closed under direct limits since $\Tor$ commutes with direct limits,
it follows that $\lim \Hom_S(C,F_i)$ is weak flat. By \cite[Lemma 2.7]{GT12}, we have the isomorphism
$$\lim\Hom_S(C,F_i)\cong\Hom_S(C,\lim F_i).$$ Then $\Hom_S(C,\lim F_i)$ is a weak flat left $R$-module.
Therefore, we have $\lim F_i\in \mathcal{WF}_C(S)$ by Proposition \ref{2.4}(1). Thus the class $\mathcal{WF}_C(S)$ is closed under direct limits.

As a similar argument to the above, we can deduce that the class $\mathcal{WI}_C(R)$ is closed under direct summands,
direct products, direct sums and direct limits. \end{proof}

The following lemma is stated in \cite{TW10} for a commutative ring, but the proof is valid in
the present context.

\begin{lemma}\label{2.9} (\cite[Theorem 2.8]{TW10}) {The following statements hold.

$\mathrm{(1)}$ If $M\in \Mod S$, then $M\in \mathcal{B}_C(S)$ if and only if $\Hom_S(C,M)\in\mathcal{A}_C(R)$.

$\mathrm{(2)}$ If $M\in \Mod R$, then $M\in \mathcal{A}_C(R)$ if and only if $C\otimes_RM\in\mathcal{B}_C(S)$. }
\end{lemma}

\begin{corollary}\label{2.10}  {The following statements hold.

$\mathrm{(1)}$ $\Hom_S(C,I)\in\mathcal{WI}_C(R)$ if and only if $I\in \mathcal{WI}(S)$.

$\mathrm{(2)}$  $C\otimes_RF\in \mathcal{WF}_C(S)$ if and only if $F\in \mathcal{WF}(R)$. }
\end{corollary}

\begin{proof} (1) ``If" part is by definition.

``Only if" part.  Since $\Hom_S(C,I)\in\mathcal{WI}_C(R)$, then we have $\Hom_S(C,I)\in \mathcal{A}_C(R)$ by Proposition \ref{2.4}(2),
and so $I\in\mathcal{B}_C(S)$ by Lemma \ref{2.9}(1).
On the other hand, there exists a weak injective left $S$-module $I'$ such that $\Hom_S(C,I)=\Hom_S(C,I')$.
Also, we have $I'\in\mathcal{B}_C(S)$ by Theorem \ref{2.2}.
Thus we have the following isomorphism: $$I\cong C\otimes_R\Hom_S(C,I)\cong C\otimes_R\Hom_S(C,I')\cong I'.$$ Hence $I$ is a weak injective left $S$-module, as desired.

(2) Similar to the proof of (1). \end{proof}

 It was shown in \cite[Propositions 2.7(2) and 2.10(2)]{BGH14} that the classes of weak flat and weak injective modules are both closed under pure submodules and pure quotients. Here we have

\begin{proposition}\label{2.11}  {The following statements hold.

$\mathrm{(1)}$  The class $\mathcal{WF}_C(S)$ is closed under pure submodules and pure quotients.

$\mathrm{(2)}$  The class $\mathcal{WI}_C(R)$ is closed under pure submodules and pure quotients. }
\end{proposition}

\begin{proof} (1)
Let $$\mathbb{Y}=\ \ 0\ra Y_1\ra Y_2\ra Y_3\ra 0$$ be a pure exact sequence in $\Mod S$ with $Y_2\in \mathcal{WF}_C(S)$.
We will show that $Y_1,Y_3\in\mathcal{WF}_C(S)$.  Since $C$ is a finitely presented left $S$-module, the sequence
$$\Hom_S(C,\mathbb{Y})=\ \ 0\lra \Hom_S(C,Y_1)\lra \Hom_S(C,Y_2)\lra \Hom_S(C,Y_3)\lra 0$$ is exact in $\Mod R$.
We claim that $\Hom_S(C,\mathbb{Y})$ is pure exact. Let $Q$ be a finitely presented left $R$-module. It is easy to check that $C\otimes_RQ$ is a finitely presented left $S$-module. By the natural isomorphism
$$\Hom_R(Q,\Hom_S(C,\mathbb{Y}))\cong\Hom_S(C\otimes_RQ,\mathbb{Y}).$$ Then we have $\Hom_S(C\otimes_RQ,\mathbb{Y})$
 is an exact sequence since $\mathbb{Y}$ is pure exact.
It follows that $\Hom_R(Q,\Hom_S(C,\mathbb{Y}))$ is exact, and hence $\Hom_S(C,\mathbb{Y})$ is a pure exact sequence.

 In the pure exact sequence $\Hom_S(C,\mathbb{Y})$, the module $\Hom_S(C,Y_2)$ is a weak flat left $R$-module by Proposition \ref{2.4}(1) since $Y_2$ is $C$-weak flat. Note that the class of weak flat modules is closed under pure submodules and pure quotients, it follows that $\Hom_S(C,Y_1)$ and $\Hom_S(C,Y_3)$ are weak flat left $R$-modules. Hence $\Hom_S(C,Y_1)$ and $\Hom_S(C,Y_3)$ belong to $\mathcal{A}_C(R)$ by Theorem \ref{2.2}. By Lemma \ref{2.9}(1), we have $Y_1,Y_3$ belong to $\mathcal{B}_C(S)$.
Consequently, $Y_1,Y_3\in \mathcal{WF}_C(S)$ by Proposition \ref{2.4}(1).

(2) By analogy with the proof of (1), one can deduce that the class $\mathcal{WI}_C(R)$ is closed under pure submodules and pure quotients. \end{proof}

\begin{theorem}\label{2.12}  {The following statements hold.

$\mathrm{(1)}$ The class $\mathcal{WF}_C(S)$ is covering and preenveloping.

$\mathrm{(2)}$ The class $\mathcal{WI}_C(R)$ is covering and preenveloping.}
\end{theorem}

\begin{proof} (1) Since the class $\mathcal{WF}_C(S)$ is closed under pure quotients by Proposition \ref{2.11} and is closed under direct sums by Proposition \ref{2.8}, one gets directly that $\mathcal{WF}_C(S)$ is covering by \cite[Theorem 2.5]{HJ08}. On the other hand, it follows from Proposition \ref{2.11} and \cite[Proposition 3.2]{HJ08} that $\mathcal{WF}_C(S)$ is a Kaplansky class. Also, the class $\mathcal{WF}_C(S)$ is closed under direct limits by Proposition \ref{2.8}. Therefore $\mathcal{WF}_C(S)$ is preenveloping by \cite[Theorem 2.5]{EL02}.

(2) The proof is similar to that of (1). \end{proof}

\section{Foxby equivalence}

In this section we investigate Foxby equivalence relative to $C$-weak injective and $C$-weak flat modules. Some known results in \cite{HW07} are generalized.

\begin{proposition}\label{3.1} {There are equivalences of categories
$$\xymatrix@C=80pt{\mathcal{WF}(R) \ar@<+4pt>[r]^{C\otimes_R-} & \mathcal{WF}_C(S) \ar@<+5pt>[l]^{\Hom_S(C,-)}_{\sim}\\
\mathcal{WI}_C(R) \ar@<+4pt>[r]^{C\otimes_R-} & \mathcal{WI}(S) \ar@<+5pt>[l]^{\Hom_S(C,-)}_{\sim}. }
$$}
\end{proposition}

\begin{proof}  It suffices to prove the first assertion. Dually, we get the second one.

We have that the functor $C\otimes_R-$ maps $\mathcal{WF}(R)$ to $\mathcal{WF}_C(S)$ by definition, and
the functor $\Hom_S(C,-)$ maps $\mathcal{WF}_C(S)$ to $\mathcal{WF}(R)$ by Proposition \ref{2.4}(1).
On the other hand, if $M\in\mathcal{WF}(R)$ and $N\in\mathcal{WF}_C(S)$, then $M\in \mathcal{A}_C(R)$ by Theorem \ref{2.2} and $N\in\mathcal{B}_C(S)$ by Proposition \ref{2.4}(1).
Then there exist natural isomorphisms: $M\cong \Hom_S(C,C\otimes_RM)$ and $N\cong C\otimes_R\Hom_S(C,M)$. Thus the assertion follows. \end{proof}

Let $n$ be a non-negative integer. For convenience, we set
\[\begin{array}{ll}\vspace{0.1cm}
\mathcal{WF}(R)_{\leq n}=\mbox{the class of left $R$-modules of weak flat dimension at most $n$,}\\\vspace{0.1cm}
\mathcal{WI}(S)_{\leq n}=\mbox{the class of left $S$-modules of weak injective dimension at most $n$,}\\\vspace{0.1cm}
\mathcal{WF}_C(S)_{\leq n}=\mbox{the class of left $S$-modules of $C$-weak flat dimension at most $n$,}\\\vspace{0.1cm}
\mathcal{WI}_C(R)_{\leq n}=\mbox{the class of left $R$-modules of $C$-weak injective dimension at most $n$.}\\
 \end{array}\]

\begin{proposition}\label{3.2} {There are equivalences of categories
$$\xymatrix@C=80pt{\mathcal{WF}(R)_{\leq n} \ar@<+4pt>[r]^{C\otimes_R-} & \mathcal{WF}_C(S)_{\leq n} \ar@<+5pt>[l]^{\Hom_S(C,-)}_{\sim}\\
\mathcal{WI}_C(R)_{\leq n} \ar@<+4pt>[r]^{C\otimes_R-} & \mathcal{WI}(S)_{\leq n} \ar@<+5pt>[l]^{\Hom_S(C,-)}_{\sim}. }
$$}
\end{proposition}

\begin{proof}  We only prove the first assertion, and the second one is dual.

The case $n=0$ holds by Proposition \ref{3.1}. Now suppose that $n\geq 1$ and $M\in\mathcal{WF}(R)_{\leq n}$. Then there exists an exact sequence
$$0\lra F_n\overset{f_n}\lra \cdots\overset{f_2}\lra F_1\overset{f_1}\lra F_0\lra M\lra 0\eqno{(3.1)}$$
in $\Mod R$ with each $F_i$ weak flat. By Corollary \ref{2.3}, we have $\coker(f_i)\in\mathcal{A}_C(R)$ for $1\leq i\leq n$.
Applying the functor $C\otimes_R-$ to (3.1), we get an exact sequence
$$0\lra C\otimes_RF_n\lra \cdots\lra C\otimes_RF_1\lra C\otimes_RF_0\lra C\otimes_RM\lra 0$$ in $\Mod S$.
Note that each $C\otimes_RF_i$ is $C$-weak flat, it follows that
$C$-$\wfd_S(C\otimes_RM)\leq n$, and thus $C\otimes_RM\in\mathcal{WF}_C(S)_{\leq n}$.

Conversely, assume that $M\in\mathcal{WF}_C(S)_{\leq n}$. Then there is an exact sequence
 $$0\lra C\otimes_RQ_n\overset{1_C\otimes_Rf_n}\lra \cdots\overset{1_C\otimes_Rf_2}\lra C\otimes_RQ_1\overset{1_C\otimes_Rf_1}\lra C\otimes_RQ_0\lra M\lra 0\eqno{(3.2)}$$ in $\Mod S$ with each $Q_i$ a weak flat left $R$-module.
Because $Q_i\in\mathcal{A}_C(R)$ by Theorem \ref{2.2}, we have each $C\otimes_RQ_i\in \mathcal{B}_C(S)$ by \cite[Proposition 4.1]{HW07}.
 It follows from \cite[Corollary 6.3]{HW07} that all $\coker(1_C\otimes_Rf_i)$ in (3.2) are in $\mathcal{B}_C(S)$.
So we get the following exact sequence
 $$0\ra \Hom_S(C,C\otimes_RQ_n)\ra \cdots\ra  \Hom_S(C,C\otimes_RQ_0)\ra \Hom_S(C,M)\ra 0.$$
Note that $\mu_{Q_i}:Q_i\lra \Hom_S(C,C\otimes_RQ_i)$ is an isomorphism for any $1\leq i\leq n$, which gives rise to the exactness of
 $$0\lra Q_n\lra \cdots\lra Q_1\lra Q_0\lra \Hom_S(C,M)\lra 0.$$ Thus $\Hom_S(C,M)\in\mathcal{WF}(R)_{\leq n}$, as desired. \end{proof}

\begin{proposition}\label{3.3}  {For any integer $n\geq 0$, the following statements hold.

$\mathrm{(1)}$ If $C$-$\wfd_S(M)\leq n$, then $M\in \mathcal{B}_C(S)$.

$\mathrm{(2)}$ If $C$-$\wid_R(M)\leq n$, then $M\in \mathcal{A}_C(R)$.}
\end{proposition}

\begin{proof} It suffices to prove the first assertion. Dually, we get the second one.

If $n=0$, then the assertion follows by Proposition \ref{2.4}(1). Now suppose $n\geq 1$, then there exists an exact sequence
$$0\ra C\otimes_RF_n\ra\cdots\ra C\otimes_RF_1\ra C\otimes_RF_0\ra M\ra 0 \eqno{(3.3)}$$ in $\Mod S$ with each $F_i$ a weak flat left $R$-module.
Since $C\otimes_RF_i\in\mathcal{B}_C(S)$ for any $0\leq i\leq n$, we have every cokernel in (3.3) is in $\mathcal{B}_C(S)$ by \cite[Theorem 6.2]{HW07}. Thus $M\in\mathcal{B}_C(S)$.\end{proof}

The following theorem is one of main results in this paper.

\begin{theorem}\label{3.4} {\rm (Foxby Equivalence)} {There are equivalences of categories
$$\xymatrix@R=20pt@C=60pt{
\mathcal{WF}(R) \ar@<+4pt>[r]^{C\otimes_R-}\ar@{^{(}->}[d] & \mathcal{WF}_C(S) \ar@<+5pt>[l]^{\Hom_S(C,-)}_{\sim}\ar@{^{(}->}[d]\\
\mathcal{WF}(R)_{\leq n} \ar@<+4pt>[r]^{C\otimes_R-}\ar@{^{(}->}[d] & \mathcal{WF}_C(S)_{\leq n} \ar@<+5pt>[l]^{\Hom_S(C,-)}_{\sim}\ar@{^{(}->}[d]\\
\mathcal{A}_C(R) \ar@<+4pt>[r]^{C\otimes_R-} & \mathcal{B}_C(S) \ar@<+5pt>[l]^{\Hom_S(C,-)}_{\sim}\\
\mathcal{WI}_C(R)_{\leq n} \ar@<+4pt>[r]^{C\otimes_R-}\ar@{^{(}->}[u] & \mathcal{WI}(S)_{\leq n} \ar@<+5pt>[l]^{\Hom_S(C,-)}_{\sim}\ar@{^{(}->}[u]\\
\mathcal{WI}_C(R)           \ar@<+4pt>[r]^{C\otimes_R-}\ar@{^{(}->}[u] & \mathcal{WI}(S).          \ar@<+5pt>[l]^{\Hom_S(C,-)}_{\sim}\ar@{^{(}->}[u]\\}
$$}
\end{theorem}

\begin{proof} This follows directly from Propositions \ref{3.1}, \ref{3.2} and \ref{3.3}. \end{proof}

\begin{proposition}\label{3.5}  {The following equalities hold.

$\mathrm{(1)}$  $\wfd_R(M)=C$-$\wfd_S(C\otimes_RM)$ for any left $R$-module $M$.

$\mathrm{(2)}$  $\wid_S(M)=C$-$\wid_R(\Hom_S(C,M))$ for any left $S$-module $M$.

$\mathrm{(3)}$ $C$-$\wid_R(M)=\wid_S(C\otimes_RM)$ for any left $R$-module $M$.

$\mathrm{(4)}$ $C$-$\wfd_S(M)=\wfd_R(\Hom_S(C,M))$ for any left $S$-module $M$.}
\end{proposition}

\begin{proof} We only prove (1), and (2)--(4) can be proved similarly.

Let $\wfd_R(M)=t<\infty$. Then $M\in\mathcal{A}_C(R)$ by Corollary \ref{2.3}, and so $\Tor_i^R(C,M)=0$ for all $i\geq 1$.
On the other hand, there is a weak flat resolution of $M$:
$$0\ra F_t\ra\cdots\ra F_1\ra F_0\ra M\ra 0,$$ which gives rise to the exactness of
$$0\ra C\otimes_RF_t\ra\cdots\ra C\otimes_RF_1\ra C\otimes_RF_0\ra C\otimes_RM\ra 0,$$
where each $C\otimes_RF_i$ is $C$-weak flat. This implies that $C$-$\wfd_S(C\otimes_RM)\leq t=\wfd_R(M)$.

Conversely, assume that $C$-$\wfd_S(C\otimes_RM)=s<\infty$. Then, by Proposition \ref{3.3}(1), we have $C\otimes_RM\in \mathcal{B}_C(S)$.
It follows that $M\in\mathcal{A}_C(R)$ by Proposition \ref{2.4}(2), which implies that  $\mu_{_{M}}:M\ra \Hom_S(C,C\otimes_RM)$ is an isomorphism.
On the other hand, there exists an exact sequence
$$\mathbb{X}=\ \ 0\ra C\otimes_RF_s\ra C\otimes_RF_{s-1}\ra\cdots\ra C\otimes_RF_0\ra C\otimes_RM\ra 0$$
in $\Mod S$ with each $F_i$ a weak flat left $R$-module. Since $C\otimes_RF_i\in\mathcal{B}_C(S)$ for any $0\leq i\leq s$,
we have that $\Hom_S(C,\mathbb{X})$ is exact by \cite[Corollary 6.3]{HW07}. Consider the following commutative diagram with the lower row exact:
$${\small\xymatrix@R=20pt@C=10pt{0 \ar[r] & F_s \ar[d]^{\cong} \ar[r] & \cdots\ar[r] & F_0 \ar[d]^{\cong} \ar[r] & M\ar[d]^{\cong} \ar[r] & 0\\
 0 \ar[r] & \Hom_S(C,C\otimes_RF_s) \ar[r] & \cdots\ar[r] &\Hom_S(C,C\otimes_RF_0)  \ar[r]  &\Hom_S(C,C\otimes_RM)\ar[r] & 0.}}$$
Then the upper row is exact in $\Mod R$ with each $F_i$ weak flat, and hence $\wfd_R(M)\leq s$.
\end{proof}

We finish this section with some applications of Theorem \ref{3.4}, which is of independent interest. For convenience, we assume that $R$ is a commutative ring.

In view of Theorem \ref{3.4}, it is natural to ask whether there is a relationship between the subcategory $\mathcal{WI}_C(R)$ and the subcategory of $\mathcal{WF}_C(S)$.
The answer is positive, that is, a suitable functor $\Hom_R(-,E)$ is discovered for any injective $R$-module $E$. Here we have

\begin{proposition}\label{3.6}  {
Let $R$ be a commutative ring and $E$ an injective $R$-module. Then

$\mathrm{(1)}$ $W\in \mathcal{WI}_C(R)$ implies $\Hom_R(W,E)\in
\mathcal{WF}_C(R)$.

$\mathrm{(2)}$ $W\in \mathcal{WF}_C(R)$ implies $\Hom_R(W,E)\in
\mathcal{WI}_C(R)$.}
\end{proposition}

\begin{proof}  (1) Let $W\in \mathcal{WI}_C(R)$. Then there exists a weak injective $R$-module $I$ such that $W=\operatorname{Hom}_R(C,I)$. Let $E$ be an injective $R$-module. One can easily check that $\operatorname{Hom}_R(I,E)$ is weak flat by the isomorphism (\cite[Lemma 1.2]{HW07}): $$\mbox{Tor}^R_1(N,\mbox{Hom}_R(I,E))\cong\mbox{Hom}_R(\mbox{Ext}^1_R(N,I),E)$$ for any super finitely presented $R$-module $N$. Thus we have
$$ \operatorname{Hom}_R(W,E)=\operatorname{Hom}_R(\operatorname{Hom}_R(C,I),E)\cong C\otimes_R\operatorname{Hom}_R(I,E).$$
It follows that $\operatorname{Hom}_R(W,E)$ is $C$-weak flat by definition.

(2) Let $W\in \mathcal{WF}_C(R)$. Then $W=C\otimes_RF$ for some weak flat $R$-module $F$. For any injective $R$-module $E$, one easily gets that $\operatorname{Hom}_R(F,E)$ is weak injective by the isomorphism (\cite[Lemma 2.16]{GT12}): $$\mbox{Ext}^1_R(N,\operatorname{Hom}_R(F,E))\cong\mbox{Hom}_R(\mbox{Tor}^R_1(N,F),E).$$ Thus we have
$$\operatorname{Hom}_R(W,E)=\operatorname{Hom}_R(C\otimes_RF,E)\cong \operatorname{Hom}_R(C,\operatorname{Hom}_R(F,E)).$$
Therefore $\operatorname{Hom}_R(W,E)$ is $C$-weak injective, as desired.\end{proof}

For convenience, we will denote by ${_R(-,-)}:=\Hom_R(-,-)$ in the following commutative diagrams. By Theorem \ref{3.4} and Proposition \ref{3.6}, we have the following two commutative diagrams for any injective $R$-module $E$:
$${\small
\xymatrix@C=1.5cm{
  \mathcal{WI}_C(R) \ar[r]^{{_R(-,E)}}\ar[d]_{C\otimes_R-} & \mathcal{WF}_C(R)\ar[d]^{{_R(C,-)}}\\
   \mathcal{WI}(R) \ar[r]^{{_R(-,E)}} & \mathcal{WF}(R), }
   \xymatrix@C=1.5cm{
  \mathcal{WF}_C(R) \ar[r]^{{_R(-,E)}}\ar[d]_{{_R(C,-)}} & \mathcal{WI}_C(R)\ar[d]^{C\otimes_R-}\\
   \mathcal{WF}(R) \ar[r]^{{_R(-,E)}} & \mathcal{WI}(R) }}$$
which give rise to the commutative diagram for all injective $R$-modules $E$ and $E'$:
$${\small
\xymatrix@C=1.5cm{
  \mathcal{WI}_C(R) \ar[r]^{{_R(-,E)}}\ar[d]_{C\otimes_R-} & \mathcal{WF}_C(R)\ar[r]^{{_R(-,E')}}&\mathcal{WI}_C(R)\\
   \mathcal{WI}(R) \ar[r]^{{_R(-,E)}}\ar[d]_{{_R(C,-)}} & \mathcal{WF}(R) \ar[r]^{{_R(-,E')}}&\mathcal{WI}(R)\ar[u]_{{_R(C,-)}}\\
   \mathcal{WI}_C(R) \ar[r]^{{_R(-,E)}}& \mathcal{WF}_C(R)\ar[r]^{{_R(-,E')}}&\mathcal{WI}_C(R).\ar[u]_{C\otimes_R-}
   }}$$

\begin{proposition}\label{3.7} {
Let $R$ be a commutative ring and $F$ a flat $R$-module. Then

$\mathrm{(1)}$ $W\in \mathcal{WI}_C(R)$ implies $W\otimes_RF\in
\mathcal{WI}_C(R)$.

$\mathrm{(2)}$ $W\in \mathcal{WF}_C(R)$ implies $W\otimes_RF\in
\mathcal{WF}_C(R)$. } \end{proposition}

\begin{proof}  (1) Let $W\in \mathcal{WI}_C(R)$. Then there exists a weak injective $R$-module $I$ such that $W=\operatorname{Hom}_R(C,I)$. For any super finitely presented $R$-module $N$, we have that $I\otimes_RF$ is weak injective by the isomorphism (\cite[Lemma 1.1]{HW07}): $$\mbox{Ext}_R^1(N,I\otimes_RF)\cong\mbox{Ext}^1_R(N,I)\otimes_RF.$$
On the other hand, by the tensor evaluation morphism (\cite[1.10]{HW07}), one easily gets that $$W\otimes_RF=\operatorname{Hom}_R(C,I)\otimes_RF\cong \operatorname{Hom}_R(C,I\otimes_RF)$$since $C$ is a semidualizing module.
Hence $W\otimes_RF$ is $C$-weak injective by definition.

(2) Let $W\in \mathcal{WF}_C(R)$. Then we have $W=C\otimes_RQ$ for some weak flat $R$-module $Q$. For any super finitely presented $R$-module $N$, one gets that $Q\otimes_RF$ is weak flat by the isomorphism (\cite[Theorem 9.48]{Ro79}): $$\Tor^R_1(N,Q\otimes_RF)\cong\Tor_1^R(N,Q)\otimes_RF.$$
On the other hand, we have that $$W\otimes_RF=(C\otimes_RQ)\otimes_RF\cong C\otimes_R(Q\otimes_RF).$$
Therefore $W\otimes_RF$ is $C$-weak flat, as desired.  \end{proof}

By Theorem \ref{3.4} and Proposition \ref{3.7}, we have the following commutative diagrams for any flat $R$-module $F$:
$${\small
\xymatrix@C=1.5cm{
  \mathcal{WI}_C(R) \ar[r]^{-\otimes_RF}\ar[d]_{C\otimes_R-} & \mathcal{WI}_C(R)\ar[d]^{C\otimes_R-}\\
   \mathcal{WI}(R) \ar[r]^{-\otimes_RF} & \mathcal{WI}(R), }
   \xymatrix@C=1.5cm{
  \mathcal{WF}_C(R) \ar[r]^{-\otimes_RF}\ar[d]_{{_R(C,-)}} & \mathcal{WF}_C(R)\ar[d]^{{_R(C,-)}}\\
   \mathcal{WF}(R) \ar[r]^{-\otimes_RF} & \mathcal{WF}(R). }}$$

Moreover, for any flat $R$-module $F$ and any injective $R$-module $E$, we obtain the following commutative diagrams:
$${\scriptsize
\xymatrix@C=15pt{
  \mathcal{WI}_C(R) \ar[r]^{-\otimes_RF}\ar[d]_{C\otimes_R-} & \mathcal{WI}_C(R)\ar[r]^{{_R(-,E)}}&\mathcal{WF}_C(R)\\
   \mathcal{WI}(R) \ar[r]^{-\otimes_RF}\ar[d]_{{_R(C,-)}} & \mathcal{WI}(R) \ar[r]^{{_R(-,E)}}&\mathcal{WF}(R)\ar[u]_{C\otimes_R-}\\
   \mathcal{WI}_C(R) \ar[r]^{-\otimes_RF}& \mathcal{WI}_C(R)\ar[r]^{{_R(-,E)}}&\mathcal{WF}_C(R),\ar[u]_{{_R(C,-)}}
   }
\!\!\!\xymatrix@C=15pt{
  \mathcal{WF}_C(R) \ar[r]^{{_R(-,E)}}\ar[d]_{{_R(C,-)}} & \mathcal{WI}_C(R)\ar[r]^{-\otimes_RF}&\mathcal{WI}_C(R)\\
   \mathcal{WF}(R) \ar[r]^{{_R(-,E)}}\ar[d]_{C\otimes_R-} & \mathcal{WI}(R) \ar[r]^{-\otimes_RF}&\mathcal{WI}(R)\ar[u]_{{_R(C,-)}}\\
   \mathcal{WF}_C(R) \ar[r]^{{_R(-,E)}}& \mathcal{WI}_C(R)\ar[r]^{-\otimes_RF}&\mathcal{WI}_C(R),\ar[u]_{C\otimes_R-}
   }}$$
   $${\scriptsize
\xymatrix@C=15pt{
  \mathcal{WF}_C(R) \ar[r]^{-\otimes_RF}\ar[d]_{{_R(C,-)}} & \mathcal{WF}_C(R)\ar[r]^{{_R(-,E)}}&\mathcal{WI}_C(R)\\
   \mathcal{WF}(R) \ar[r]^{-\otimes_RF}\ar[d]_{C\otimes_R-} & \mathcal{WF}(R) \ar[r]^{{_R(-,E)}}&\mathcal{WI}(R)\ar[u]_{{_R(C,-)}}\\
   \mathcal{WF}_C(R) \ar[r]^{-\otimes_RF}& \mathcal{WF}_C(R)\ar[r]^{{_R(-,E)}}&\mathcal{WI}_C(R),\ar[u]_{C\otimes_R-}
   }
\!\!\!\xymatrix@C=15pt{
  \mathcal{WI}_C(R) \ar[r]^{{_R(-,E)}}\ar[d]_{C\otimes_R-} & \mathcal{WF}_C(R)\ar[r]^{-\otimes_RF}&\mathcal{WF}_C(R)\\
   \mathcal{WI}(R) \ar[r]^{{_R(-,E)}}\ar[d]_{{_R(C,-)}} & \mathcal{WF}(R) \ar[r]^{-\otimes_RF}&\mathcal{WF}(R)\ar[u]_{C\otimes_R-}\\
   \mathcal{WI}_C(R) \ar[r]^{{_R(-,E)}}& \mathcal{WF}_C(R)\ar[r]^{-\otimes_RF}&\mathcal{WF}_C(R).\ar[u]_{{_R(C,-)}}
   }}$$

\begin{proposition}\label{3.8} {
Let $R$ be a commutative ring and $P$ a projective $R$-module. Then

$\mathrm{(1)}$ $W\in \mathcal{WF}_C(R)$ implies $\Hom_R(P,W)\in
\mathcal{WF}_C(R)$.

$\mathrm{(2)}$ $W\in \mathcal{WI}_C(R)$ implies $\Hom_R(P,W)\in
\mathcal{WI}_C(R)$.}
\end{proposition}

\begin{proof}
(1) Let $W\in \mathcal{WF}_C(R)$. Then there exists a weak flat $R$-module $F$ such that $W=C\otimes_RF$. For any super finitely presented $R$-module $N$, we have an exact sequence $$0\rightarrow N'\rightarrow P_0\rightarrow N\rightarrow 0$$ with $P_0$ finitely generated projective and $N'$ super finitely presented. Now consider the following commutative diagram
$$\xymatrix@C=15pt{0\ar[r] & \operatorname{Hom}_R(P,\operatorname{Tor}_1^R(N,F))\ar[r]\ar[d] & \operatorname{Hom}_R(P,N'\otimes_RF) \ar[r]\ar[d]^\cong & \operatorname{Hom}_R(P,P_0\otimes_RF)\ar[d]^\cong\\
0\ar[r] & \operatorname{Tor}^R_1(\operatorname{Hom}_R(P,F),N)\ar[r] & \operatorname{Hom}_R(P,F)\otimes_RN' \ar[r] & \operatorname{Hom}_R(P,F)\otimes_RP_0.
}$$
Since $N'$ and $P_0$ are finitely generated and $P$ is projective, one gets that the right two morphisms are isomorphic by \cite[Appendix A, Lemma 1.4]{SSW}. Hence we have
 $$\mbox{Tor}^R_1(\mbox{Hom}_R(P,F),N)\cong \mbox{Hom}_R(P,\operatorname{Tor}^R_1(N,F))=0$$ since $F$ is weak flat.  It follows that
 $\operatorname{Hom}_R(P,F)$ is weak flat. On the other hand, we have
$$\operatorname{Hom}_R(P,W)=\operatorname{Hom}_R(P,C\otimes_RF)\cong \operatorname{Hom}_R(P,F)\otimes_RC\cong C\otimes_R\operatorname{Hom}_R(P,F).$$
Thus $\operatorname{Hom}_R(P,W)$ is $C$-weak flat.

(2)  Let $W\in \mathcal{WI}_C(R)$. Then $W=\operatorname{Hom}_R(C,I)$ for some weak injective $R$-module $I$. For any super finitely presented $R$-module $N$, we have an exact sequence $$0\rightarrow N'\rightarrow P_0\rightarrow N\rightarrow 0$$ with $P_0$ finitely generated projective. Consider the following commutative diagram
$${\small\xymatrix@C=10pt{ \operatorname{Hom}_R(P,\operatorname{Hom}_R(P_0,I))\ar[r]\ar[d]^\cong & \operatorname{Hom}_R(P,\operatorname{Hom}_R(N',I)) \ar[r]\ar[d]^\cong & \operatorname{Hom}_R(P,\operatorname{Ext}^1_R(N,I))\ar[d]\ar[r] &0\\
 \operatorname{Hom}_R(P_0,\operatorname{Hom}_R(P,I))\ar[r] & \operatorname{Hom}_R(N',\operatorname{Hom}_R(P,I)) \ar[r] & \operatorname{Ext}^1_R(N,\operatorname{Hom}_R(P,I))\ar[r]&0.
}}$$
Then the left two morphisms in the above diagram are isomorphic viewed as the swap maps. It follows that $\operatorname{Ext}^1_R(N,\operatorname{Hom}_R(P,I))\cong\operatorname{Hom}_R(P,\operatorname{Ext}^1_R(N,I))$. Note that $I$ is weak injective, we get that $\operatorname{Hom}_R(P,I)$ is weak injective.
\end{proof}

By Theorem \ref{3.4} and Proposition \ref{3.8},  we have the following commutative diagrams for any projective $R$-module $P$:
$${\small
\xymatrix@C=1.5cm{
  \mathcal{WF}_C(R) \ar[r]^{{_R(P,-)}}\ar[d]_{{_R(C,-)}} & \mathcal{WF}_C(R)\ar[d]^{{_R(C,-)}}\\
   \mathcal{WF}(R) \ar[r]^{{_R(P,-)}} & \mathcal{WF}(R), }
   \xymatrix@C=1.5cm{
  \mathcal{WI}_C(R) \ar[r]^{{_R(P,-)}}\ar[d]_{C\otimes_R-} & \mathcal{WI}_C(R)\ar[d]^{C\otimes_R-}\\
   \mathcal{WI}(R) \ar[r]^{{_R(P,-)}} & \mathcal{WI}(R). }}$$
Consequently, we obtain the following commutative diagrams for any injective $R$-module $E$, any flat $R$-module $F$ and any projective $R$-module $P$:
$${\scriptsize
  \xymatrix@C=15pt{
  \mathcal{WF}_C(R) \ar[r]^{{_R(P,-)}}\ar[d]_{{_R(C,-)}} & \mathcal{WF}_C(R)\ar[r]^{{_R(-,E)}}&\mathcal{WI}_C(R)\\
   \mathcal{WF}(R) \ar[r]^{{_R(P,-)}}\ar[d]_{C\otimes_R-} & \mathcal{WF}(R) \ar[r]^{{_R(-,E)}}&\mathcal{WI}(R)\ar[u]_{{_R(C,-)}}\\
   \mathcal{WF}_C(R) \ar[r]^{{_R(P,-)}}& \mathcal{WF}_C(R)\ar[r]^{{_R(-,E)}}&\mathcal{WI}_C(R),\ar[u]_{C\otimes_R-}
   }
  \!\!\! \xymatrix@C=15pt{
  \mathcal{WI}_C(R) \ar[r]^{{_R(-,E)}}\ar[d]_{C\otimes_R-} & \mathcal{WF}_C(R)\ar[r]^{{_R(P,-)}}&\mathcal{WF}_C(R)\\
   \mathcal{WI}(R) \ar[r]^{{_R(-,E)}}\ar[d]_{{_R(C,-)}} & \mathcal{WF}(R) \ar[r]^{{_R(P,-)}}&\mathcal{WF}(R)\ar[u]_{C\otimes_R-}\\
   \mathcal{WI}_C(R) \ar[r]^{{_R(-,E)}}& \mathcal{WF}_C(R)\ar[r]^{{_R(P,-)}}&\mathcal{WF}_C(R),\ar[u]_{{_R(C,-)}}
   }}$$
 $${\scriptsize
  \xymatrix@C=15pt{
  \mathcal{WF}_C(R) \ar[r]^{{_R(P,-)}}\ar[d]_{{_R(C,-)}} & \mathcal{WF}_C(R)\ar[r]^{-\otimes_RF}&\mathcal{WF}_C(R)\\
   \mathcal{WF}(R) \ar[r]^{{_R(P,-)}}\ar[d]_{C\otimes_R-} & \mathcal{WF}(R) \ar[r]^{-\otimes_RF}&\mathcal{WF}(R)\ar[u]_{C\otimes_R-}\\
   \mathcal{WF}_C(R) \ar[r]^{{_R(P,-)}}& \mathcal{WF}_C(R)\ar[r]^{-\otimes_RF}&\mathcal{WF}_C(R),\ar[u]_{{_R(C,-)}}
   }
  \!\!\! \xymatrix@C=15pt{
  \mathcal{WF}_C(R) \ar[r]^{-\otimes_RF}\ar[d]_{{_R(C,-)}} & \mathcal{WF}_C(R)\ar[r]^{{_R(P,-)}}&\mathcal{WF}_C(R)\\
   \mathcal{WF}(R) \ar[r]^{-\otimes_RF}\ar[d]_{C\otimes_R-} & \mathcal{WF}(R) \ar[r]^{{_R(P,-)}}&\mathcal{WF}(R)\ar[u]_{C\otimes_R-}\\
   \mathcal{WF}_C(R) \ar[r]^{-\otimes_RF}& \mathcal{WF}_C(R)\ar[r]^{{_R(P,-)}}&\mathcal{WF}_C(R),\ar[u]_{{_R(C,-)}}
   }}$$
$${\scriptsize
  \xymatrix@C=15pt{
  \mathcal{WI}_C(R) \ar[r]^{{_R(P,-)}}\ar[d]_{C\otimes_R-} & \mathcal{WI}_C(R)\ar[r]^{{_R(-,E)}}&\mathcal{WF}_C(R)\\
   \mathcal{WI}(R) \ar[r]^{{_R(P,-)}}\ar[d]_{{_R(C,-)}} & \mathcal{WI}(R) \ar[r]^{{_R(-,E)}}&\mathcal{WF}(R)\ar[u]_{C\otimes_R-}\\
   \mathcal{WI}_C(R) \ar[r]^{{_R(P,-)}}& \mathcal{WI}_C(R)\ar[r]^{{_R(-,E)}}&\mathcal{WF}_C(R),\ar[u]_{{_R(C,-)}}
   }
  \!\!\! \xymatrix@C=15pt{
  \mathcal{WF}_C(R) \ar[r]^{{_R(-,E)}}\ar[d]_{{_R(C,-)}} & \mathcal{WI}_C(R)\ar[r]^{{_R(P,-)}}&\mathcal{WI}_C(R)\\
   \mathcal{WF}(R) \ar[r]^{{_R(-,E)}}\ar[d]_{C\otimes_R-} & \mathcal{WI}(R) \ar[r]^{{_R(P,-)}}&\mathcal{WI}(R)\ar[u]_{{_R(C,-)}}\\
   \mathcal{WF}_C(R) \ar[r]^{{_R(-,E)}}& \mathcal{WI}_C(R)\ar[r]^{{_R(P,-)}}&\mathcal{WI}_C(R),\ar[u]_{C\otimes_R-}
   }}$$
 $${\scriptsize
  \xymatrix@C=15pt{
  \mathcal{WI}_C(R) \ar[r]^{{_R(P,-)}}\ar[d]_{C\otimes_R-} & \mathcal{WI}_C(R)\ar[r]^{-\otimes_RF}&\mathcal{WI}_C(R)\\
   \mathcal{WI}(R) \ar[r]^{{_R(P,-)}}\ar[d]_{{_R(C,-)}} & \mathcal{WI}(R) \ar[r]^{-\otimes_RF}&\mathcal{WI}(R)\ar[u]_{{_R(C,-)}}\\
   \mathcal{WI}_C(R) \ar[r]^{{_R(P,-)}}& \mathcal{WI}_C(R)\ar[r]^{-\otimes_RF}&\mathcal{WI}_C(R),\ar[u]_{C\otimes_R-}
   }
  \!\!\! \xymatrix@C=15pt{
  \mathcal{WI}_C(R) \ar[r]^{-\otimes_RF}\ar[d]_{C\otimes_R-} & \mathcal{WI}_C(R)\ar[r]^{{_R(P,-)}}&\mathcal{WI}_C(R)\\
   \mathcal{WI}(R) \ar[r]^{-\otimes_RF}\ar[d]_{{_R(C,-)}} & \mathcal{WI}(R) \ar[r]^{{_R(P,-)}}&\mathcal{WI}(R)\ar[u]_{{_R(C,-)}}\\
   \mathcal{WI}_C(R) \ar[r]^{-\otimes_RF}& \mathcal{WI}_C(R)\ar[r]^{{_R(P,-)}}&\mathcal{WI}_C(R).\ar[u]_{C\otimes_R-}
   }}$$

\section{Stability of the Auslander and Bass classes }

In this section, we show that an iteration of the procedure used to describe the Auslander class (resp. Bass class) yields exactly the
Auslander class (resp. Bass class), which generalize \cite[Theorems 2 and 6.1]{HW07} and \cite[Propositions 3.6 and 3.7]{EH09}. This enables us to provide more beautiful characterizations of the modules in the Auslander
and Bass classes. Finally, special attention is
paid to giving a partial answer to Question 2.\vspace{0.2cm}

\begin{lemma}\label{4.1} (\cite[Theorems 2 and 6.1]{HW07}) {A left $R$-module
$M\in \mathcal{A}_C(R)$ if and only if there exists an exact sequence
$$\mathbb{X}=\ \ \cdots\lra P_1\lra P_0\lra U^0\lra U^1\lra\cdots$$ in $\Mod R$ with each $P_i$ projective (or flat)
and $U^i$ $C$-injective such that $M\cong\coker(P_1\ra P_0)$ and the complex $C\otimes_R\mathbb{X}$ is exact.

 A left $S$-module $N\in \mathcal{B}_C(S)$ if and only if there exists an exact sequence
$$\mathbb{Y}=\ \ \cdots\lra W_1\lra W_0\lra I^0\lra I^1\lra\cdots$$ in $\Mod S$ with each $I^i$ injective
and $W_i$ $C$-projective such that $N\cong\ker(I^0\ra I^1)$ and the complex $\Hom_S(C,\mathbb{Y})$ is exact.}
\end{lemma}

\begin{remark}\label{4.2}  Since ${}_SC_R$ is a faithfully
semidualizing bimodule, it is straightforward to get that all
kernels and cokernels of $\mathbb{X}$ in Lemma \ref{4.1} are in
$\mathcal{A}_C(R)$ by \cite[Lemmas 4.1, 5.1 and Corollary
6.3]{HW07}. Similarly, one can easily obtain that all kernels and
cokernels of $\mathbb{Y}$ in Lemma \ref{4.1} are in $\mathcal{B}_C(S)$.
\end{remark}

In the following, $\mathscr{A}$ denotes an abelian category, all subcategories are full subcategories of $\mathscr{A}$ closed under isomorphisms.
We fix subcategories $\mathscr{X}$ and $\mathscr{Y}$ of $\mathscr{A}$. Recall from \cite{SSW08} that a subcategory $\mathscr{X}$ of $\mathscr{Y}$ is called a \textit{generator} (resp. \textit{cogenerator}) for $\mathscr{Y}$ if for any object $Y$ in $\mathscr{Y}$, there exists an exact sequence $0\ra Y'\ra X\ra Y\ra 0$ (resp. $0\ra Y\ra X\ra Y'\ra 0$) in $\mathscr{Y}$ with $X$ an object in $\mathscr{X}$.
We use gen $\mathscr{Y}$ (resp. cogen $\mathscr{Y}$) to denote a generator (resp. cogenerator) for a subcategory $\mathscr{Y}$.\vspace{0.1cm}

The following two results play a crucial role in the section, which give a generalization of \cite[Theorem 5.3]{Hu13}.

\begin{proposition}\label{4.3} {Let $\mathscr{X}$ be closed under extensions and $$\cdots\ra G_n\ra\cdots \ra G_1\ra G_0\ra M\ra 0\eqno{(4.1)}$$ be an exact sequence in $\mathscr{A}$ with all $G_i$ objects in $\mathscr{X}$.
Then we have the following

$\mathrm{(1)}$ There exists an exact sequence
$$\cdots\ra P_n\ra \cdots\ra P_1\ra P_0\ra M\ra 0 \eqno{(4.2)}$$ in $\mathscr{A}$ with all $P_i$ objects in gen $\mathscr{X}$.

$\mathrm{(2)}$ Let $\mathscr{A}=\Mod R$, and $D$ be an object in $\mathscr{A}$ such that any short exact sequence in $\mathscr{X}$ is $D\otimes_R-$ exact.
If {\rm (4.1)} is $D\otimes_R-$ exact,
then so is {\rm (4.2)}.

$\mathrm{(3)}$ Let $D$ be an object in $\mathscr{A}$ such that any short exact sequence in $\mathscr{X}$ is $\Hom_\mathscr{A}(D,-)$ exact {\rm (}resp. $\Hom_\mathscr{A}(-,D)$ exact{\rm )}.
If {\rm (4.1)} is  $\Hom_\mathscr{A}(D,-)$ exact {\rm (}resp. $\Hom_\mathscr{A}(-,D)$ exact{\rm )},
then so is {\rm (4.2)}. }
\end{proposition}

\begin{proof} (1)\ Let $$\cdots\ra G_n\ra\cdots\ra G_1\ra G_0\ra M\ra 0$$ be an exact sequence in $\mathscr{A}$ with all $G_i\in\mathscr{X}$.
Put $K_1=\im(G_1\ra G_0)$. Then we get the following two exact sequences $$\cdots\ra G_n\ra\cdots\ra G_1\ra K_1\ra 0\ \ \mbox{ and }\ \  0\ra K_1\ra G_0\ra M\ra 0.$$
On the other hand, there exists a short exact sequence $$0\ra G \ra P_0\ra G_0\ra 0$$
 in $\mathscr{A}$ with $P_0$ an object in gen $\mathscr{X}$ and $G\in \mathscr{X}$. Consider the following pull-back diagram:
$$\xymatrix@R=16pt@C=16pt{& 0 \ar[d] & 0 \ar[d] & &\\
&  G \ar@{=}[r] \ar[d]& G \ar[d] & & \\
0 \ar[r] & N \ar[r] \ar[d] & P_0 \ar[r] \ar[d] & M \ar[r] \ar@{=}[d] & 0\\
0 \ar[r] & K_1 \ar[r] \ar[d] & G_0 \ar[r] \ar[d] & M \ar[r] & 0\\
& 0 & ~0. & &}$$
Let $K_2=\im(G_2\ra G_1)$. Then we have the following exact sequences
$$\cdots\ra G_n\ra\cdots\ra G_2\ra K_2\ra 0\ \ \mbox{ and }\ \  0\ra K_2\ra G_1\ra K_1\ra 0.$$
 Now consider the following pull-back diagram:
$$\xymatrix@R=16pt@C=16pt{& & 0 \ar[d] & 0 \ar[d]& &\\
& & G \ar@{=}[r] \ar[d] & G \ar[d]& &\\
0 \ar[r] & K_2 \ar@{=}[d] \ar[r] & G'_1 \ar[d] \ar[r] &N \ar[d] \ar[r] & 0\\
0 \ar[r] & K_2 \ar[r] & G_1 \ar[r] \ar[d] & K_1 \ar[d] \ar[r] & 0 &\\
& & 0 & ~0. & & }$$
In the sequence $0\ra G\ra G'_1\ra G_1\ra 0$, both $G$ and $G_1$ belong to $\mathscr{X}$, then so is $G'_1$.
Thus we obtain the following exact sequences $$0\ra N\ra P_0\ra M\ra 0$$ and $$\cdots\ra G_n\ra\cdots\ra G_2\ra G'_1\ra N\ra 0,\eqno{(4.3)}$$ where $P_0$ is an object in gen $\mathscr{X}$,
and $G'_1$ and $G_i\ (i=2,3,\cdots)$ belong to $\mathscr{X}$. By repeating the above step to (4.3) and so on, one gets the desired exact sequence (4.2).

(2) Let $D$ be an object in $\mathscr{A}$ such that any short exact sequence in $\mathscr{X}$ is $D\otimes_R-$ exact.
Then the middle columns in the above two diagrams are $D\otimes_R-$ exact. If $(4.1)$ is $D\otimes_R-$ exact, then both third rows in the above two diagrams are $D\otimes_R-$ exact. Thus both the middle rows in these two diagrams are also $D\otimes_R-$ exact. Hence the sequence
(4.3) is $D\otimes_R-$ exact. Continuing this process, one can easily deduce that (4.2) is $D\otimes_R-$ exact.

(3) The proof is similar to that of (2).
\end{proof}

Dually, we have the following

\begin{proposition}\label{4.4} {Let $\mathscr{X}$ be closed under extensions and let $n\geq 1$ and $$ 0\ra M\ra G^0\ra G^1\ra \cdots \ra G^n\ra\cdots\eqno{(4.4)}$$ be an exact sequence in $\mathscr{A}$ with all $G^i$ objects in $\mathscr{X}$.
Then we have the following

$\mathrm{(1)}$ There exists an exact sequence
$$ 0\ra M\ra I^0\ra I^1\ra \cdots \ra I^n\ra\cdots\eqno{(4.5)}$$ in $\mathscr{A}$ with all $I^i$ objects in cogen $\mathscr{X}$.

$\mathrm{(2)}$ Let $\mathscr{A}=\Mod R$, and  $D$ be an object in $\mathscr{A}$ such that any short exact sequence in $\mathscr{X}$ is $D\otimes_R-$ exact.
If {\rm (4.4)} is $D\otimes_R-$ exact, then so is {\rm (4.5)}.

$\mathrm{(3)}$ Let $D$ be an object in $\mathscr{A}$ such that any short exact sequence in $\mathscr{X}$ is $\Hom_\mathscr{A}(D,-)$ exact {\rm (}resp. $\Hom_\mathscr{A}(-,D)$ exact{\rm )}.
If {\rm (4.4)} is  $\Hom_\mathscr{A}(D,-)$ exact {\rm (}resp. $\Hom_\mathscr{A}(-,D)$ exact{\rm )},
then so is {\rm (4.5)}. }
\end{proposition}

By Propositions \ref{4.3} and \ref{4.4}, we get immediately the following two results.

\begin{corollary}\label{4.5}  {$\mathrm{(1)}$  A left $R$-module $M$ has a projective resolution which is $C\otimes_R-$ exact if and only if
there exists a $C\otimes_R-$ exact exact sequence $\cdots\ra G_1\ra G_0\ra M\ra 0$ with each $G_i\in\mathcal{A}_C(R)$.

 $\mathrm{(2)}$  A left $S$-module $N$ has a $C$-projective resolution which is $\Hom_R(C,-)$ exact if and only if
there exists a $\Hom_R(C,-)$ exact exact sequence $\cdots\ra V_1\ra V_0\ra N\ra 0$ with each $V_i\in\mathcal{B}_C(S)$. }
\end{corollary}

\begin{corollary}\label{4.6}  {$\mathrm{(1)}$ A left $R$-module $M$ has a $C$-injective coresolution which is $C\otimes_R-$ exact if and only if
there exists a $C\otimes_R-$ exact exact sequence $0\ra M\ra G^0\ra G^1\ra \cdots$ with each $G^i\in\mathcal{A}_C(R)$.

 $\mathrm{(2)}$ A left $S$-module $N$ has an injective coresolution which is $\Hom_R(C,-)$ exact if and only if
there exists a $\Hom_R(C,-)$ exact exact sequence $0\ra N\ra V^0\ra V^1\ra \cdots$ with each $V^i\in\mathcal{B}_C(S)$. }
\end{corollary}

Set $[\mathcal{A}_C(R)]^{1}=\mathcal{A}_C(R)$, and inductively set $[\mathcal{A}_C(R)]^{n+1}=\{M\in \Mod R\mid$ there exists a $C\otimes_R-$ exact
exact sequence $\cdots\ra G_1\ra G_0\ra G^0\ra G^1\ra\cdots$ in $\Mod R$ with all $G_i$ and $G^i$ in
$[\mathcal{A}_C(R)]^n$ such that $M\cong\coker(G_1\ra G_0)\}$ for any $n\geq 1$.

\begin{theorem}\label{4.7}
{$[\mathcal{A}_C(R)]^n=\mathcal{A}_C(R)$ for any $n\geq 1$. }
\end{theorem}

\begin{proof} It is easy to see that $\mathcal{A}_C(R)\subseteq[\mathcal{A}_C(R)]^2\subseteq[\mathcal{A}_C(R)]^3\subseteq\cdots$ is an ascending chain of subcategories of $\Mod R$.

Let $A\in[\mathcal{A}_C(R)]^2$. Then there exists a $C\otimes_R-$ exact exact sequence $$\cdots\ra G_1\ra G_0\ra G^0\ra G^1\ra\cdots$$ in $\Mod R$ with all $G_i$ and $G^i$ in
$\mathcal{A}_C(R)$ such that $A\cong\coker(G_1\ra G_0)$. It follows from Corollary \ref{4.5} that $A$ has a projective resolution which is $C\otimes_R-$ exact. On the other hand, $A$ has a $C$-injective coresolution which is $C\otimes_R-$ exact by Corollary \ref{4.6}. Hence $A\in\mathcal{A}_C(R)$ by Lemma \ref{4.1}, and so $[\mathcal{A}_C(R)]^2\subseteq\mathcal{A}_C(R)$. Thus we have that $[\mathcal{A}_C(R)]^2=\mathcal{A}_C(R)$. By using induction on $n$ we can easily get the assertion. \end{proof}

Set $[\mathcal{B}_C(S)]^{1}=\mathcal{B}_C(S)$, and inductively set $[\mathcal{B}_C(S)]^{n+1}=\{M\in \Mod S\mid$ there exists a $\Hom_S(C,-)$ exact
exact sequence $\cdots\ra V_1\ra V_0\ra V^0\ra V^1\ra\cdots$ in $\Mod S$ with all $V_i$ and $V^i$ in
$[\mathcal{B}_C(S)]^n$ such that $M\cong\ker(V^0\ra V^1)\}$ for any $n\geq 1$. Dual to Theorem \ref{4.7}, one easily gets the following result.

\begin{theorem}\label{4.8}
{$[\mathcal{B}_C(S)]^n=\mathcal{B}_C(S)$ for any $n\geq 1$. }
\end{theorem}

As applications of Theorems \ref{4.7} and \ref{4.8}, we can immediately obtain the following two results, which give some equivalent characterizations of the modules in the Auslander and Bass classes in terms of weak flat, weak injective, $C$-weak injective and $C$-weak flat modules. For a non-negative integer $n$,  we use $\mathcal{P}(R)_{\leq n}$ (resp., $\mathcal{F}(R)_{\leq n}$, $\mathcal{I}(S)_{\leq n}$ and $\mathcal{FI}(S)_{\leq n}$) to denote the class of left $R$- (or left $S$-) modules of projective
(resp., flat, injective and FP-injective) dimension at most $n$.

\begin{corollary}\label{4.9}  {The following statements are equivalent for a module $M\in\Mod R$.

$\mathrm{(1)}$ $M\in \mathcal{A}_C(R)$.

$\mathrm{(2)}$  There exists a $C\otimes_R-$ exact
exact sequence $$\cdots\ra G_1\ra G_0\ra G_{-1}\ra \cdots$$ in $\Mod R$ with all $G_i\in
\mathcal{P}(R)_{\leq n}$ or $\mathcal{I}_C(R)_{\leq n}$ such that $M\cong\coker(G_1\ra G_0)$.

$\mathrm{(3)}$  There exists a $C\otimes_R-$ exact
exact sequence $$\cdots\ra G_1\ra G_0\ra G_{-1}\ra\cdots$$ in $\Mod R$ with all $G_i\in\mathcal{F}(R)_{\leq n}$ or $\mathcal{I}_C(R)_{\leq n}$ such that $M\cong\coker(G_1\ra G_0)$.

$\mathrm{(4)}$  There exists a $C\otimes_R-$ exact
exact sequence $$\cdots\ra G_1\ra G_0\ra G_{-1}\ra\cdots$$ in $\Mod R$ with all $G_i\in\mathcal{F}(R)_{\leq n}$ or $\mathcal{FI}_C(R)_{\leq n}$ such that $M\cong\coker(G_1\ra G_0)$.

$\mathrm{(5)}$  There exists a $C\otimes_R-$ exact
exact sequence $$\cdots\ra G_1\ra G_0\ra G_{-1}\ra\cdots$$ in $\Mod R$ with all $G_i\in\mathcal{WF}(R)_{\leq n}$ or $\mathcal{WI}_C(R)_{\leq n}$ such that $M\cong\coker(G_1\ra G_0)$.}
\end{corollary}

\begin{corollary}\label{4.10}  {The following statements are equivalent for a module $N\in\Mod S$.

$\mathrm{(1)}$ $N\in \mathcal{B}_C(S)$.

$\mathrm{(2)}$  There exists a $\Hom_S(C,-)$ exact
exact sequence $$\cdots\ra V_1\ra V_0\ra V_{-1}\ra \cdots$$ in $\Mod S$ with all $V_i\in
\mathcal{I}(S)_{\leq n}$ or $\mathcal{P}_C(S)_{\leq n}$ such that $N\cong\coker(V_1\ra V_0)$.

$\mathrm{(3)}$  There exists a $\Hom_S(C,-)$ exact
exact sequence $$\cdots\ra V_1\ra V_0\ra V_{-1}\ra\cdots$$ in $\Mod R$ with all $V_i\in\mathcal{I}(S)_{\leq n}$ or $\mathcal{F}_C(S)_{\leq n}$ such that $N\cong\coker(V_1\ra V_0)$.

$\mathrm{(4)}$  There exists a $\Hom_S(C,-)$ exact
exact sequence $$\cdots\ra V_1\ra V_0\ra V_{-1}\ra\cdots$$ in $\Mod R$ with all $V_i\in\mathcal{FI}(S)_{\leq n}$ or $\mathcal{F}_C(S)_{\leq n}$ such that $N\cong\coker(V_1\ra V_0)$.

$\mathrm{(5)}$  There exists a $\Hom_S(C,-)$ exact
exact sequence $$\cdots\ra V_1\ra V_0\ra V_{-1}\ra\cdots$$ in $\Mod S$ with all $V_i\in\mathcal{WI}(S)_{\leq n}$ or $\mathcal{WF}_C(S)_{\leq n}$ such that $N\cong\coker(V_1\ra V_0)$.}
\end{corollary}

We round off the paper by giving a partial answer to Question 2, which may be viewed as an
illustration of the usefulness of the stability of the Auslander and Bass classes. Before that, recall from \cite{GW12} that a left $R$-module $M$ is called
\textit{Gorenstein FP-injective} if there exists an exact sequence of FP-injective left $R$-modules $\cdots\ra E_1\ra E_0\ra E^0\ra E^1\ra \cdots$ with $M=\ker(E^0\ra E^1)$ such that $\Hom_R(P,-)$ leaves this sequence exact whenever $P$ is a finitely presented left $R$-module of finite projective dimension. \vspace{0.1cm}

\begin{theorem}\label{4.11}    {If ${}_SC_R$ is a faithfully semidualizing bimodule with finite $S$-projective dimension, then every Gorenstein injective modules in $\Mod S$ is in $\mathcal{B}_C(S)$. }
\end{theorem}

\begin{proof}  Let ${}_SC_R$ be faithfully semidualizing with finite $S$-projective dimension. Then it is (super) finitely presented, and so every Gorenstein FP-injective module in the sense of \cite{GW12} belongs to $\mathcal{B}_C(S)$ by Corollary \ref{4.10}.
It follows that all Gorenstein injective modules in $\Mod S$ are in $\mathcal{B}_C(S)$ by \cite[Proposition 2.5]{GW12}.\end{proof}


\vspace{0.4cm}


\begin{thebibliography}{99}

\bibitem{ATY05}T. Araya, R. Takahashi, and Y. Yoshino, Homological invariants associated to semi-dualizing bimodules, J. Math. Kyoto Univ. {\bf 45}(2005), 287--306.

\bibitem{BGO16a}D. Bennis, J. R. Garc\'{i}a Rozas, and L. Oyonarte,
Relative Gorenstein dimensions, Mediterr. J. Math. {\bf 13}(2016),
65--91.

\bibitem{BGO16b}D. Bennis, J. R. Garc\'{i}a Rozas, and L. Oyonarte,   Relative
projective and injective dimensions, Comm. Algebra {\bf 44}(2016),
3383--3396.

\bibitem{BGO16c}D. Bennis, J. R. Garc\'{i}a Rozas, and L. Oyonarte,   When do Foxby
classes coincide with modules of finite Gorenstein dimension?,
Kyoto J. Math. {\bf 56}(2016), 785--802.

\bibitem{BM07}D. Bennis and N. Mahdou,  Strongly Gorenstein projective, injective, and flat modules,  J. Pure Appl. Algebra, {\bf 210}(2007), 437--445.

\bibitem{BGH14}D. Bravo, J. Gillespie, and M. Hovey,  The stable modules category of a general ring, arXiv:1405.5768v1.

\bibitem{Br82} K. S. Brown,  Cohomology of Groups, Springer-Verlag, New York, 1982.

\bibitem{Ch01}L. W. Christensen,  Semi-dualizing complexes and their Auslander categories, Trans.
Amer. Math. Soc. {\bf 353}(2001), 1839--1883.

\bibitem{EH09}E. E. Enochs and H. Holm, Cotorsion pairs asociated with Auslander categories, Israel J. Math. {\bf 174} (2009), 253--268.

\bibitem{EJ00}E. E. Enochs and O. M. G. Jenda,  Relative Homological
Algebra,  de Gruyter Expositions in Math. {\bf 30}, Walter de
Gruyter GmbH \& Co. KG, Berlin, 2000.

\bibitem{EJL05}E. E. Enochs, O. M. G. Jenda, and J. A. L\'{o}pez-Ramos,  Dualizing modules and $n$-perfect rings, Proc. Edinb. Math. Soc. {\bf 48}(2005), 75--90.

\bibitem{EL02}E. E. Enochs and J. A. L\'{o}pez-Ramos,  Kaplansky classes, Rend. Sem. Mat. Univ.
Padova {\bf 107}(2002), 67--79.

\bibitem{Fo73}H.-B. Foxby,   Gorenstein modules and related modules, Math. Scand. {\bf 31}(1973), 267--284.

\bibitem{GH15a}Z. H. Gao and Z. Y. Huang,  Weak injective covers and dimension of modules, Acta Math. Hungar. {\bf 147}(2015), 135--157.

\bibitem{GW12}Z. H. Gao and F. G. Wang, Coherent rings and Gorenstein FP-injective modules, Comm. Algebra {\bf 40}(2012), 1669--1679.

\bibitem{GW15}Z. H. Gao and F. G. Wang,  Weak injective and weak flat modules, Comm. Algebra {\bf 43}(2015), 3857--3868.

\bibitem{Go84}E. S. Golod,   G-dimension and generalized perfect ideals, Trudy Mat. Inst. Steklov {\bf 165}(1984), 62--66.

\bibitem{GT12}R. G\"{o}bel and J. Trlifaj,  Approximations and Endomorphism Algebras of Modules, de Gruyter
Expositions in Math. {\bf 41}, 2nd revised and extended edition,
Walter de Gruyter GmbH \& Co. KG, Berlin-Boston 2012.

\bibitem{HJ06}H. Holm and P. J{\o}rgensen,  Semidualizing modules and related Gorenstein
homological dimensions. J. Pure Appl. Algebra {\bf 205}(2006),
423--445.

\bibitem{HJ08}H. Holm and P. J{\o}rgensen,  Covers, precovers and purity,  Illinois J. Math. {\bf 52}(2008), 691--703.

\bibitem{HW07}H. Holm and D. White,  Foxby equivalence over associative rings, J. Math. Kyoto Univ. {\bf 47}(2007), 781--808.

\bibitem{Hu13}Z. Y. Huang,  Proper resolutions and Gorenstein categories, J. Algebra {\bf 393}(2013), 142--169.

\bibitem{HM09}L. Hummel and T. Marley,  The Auslander-Bridger formula and the Gorenstein property for coherent rings. J. Commut. Algebra {\bf
1}(2009), 283-314.

\bibitem{Ro79}J. J. Rotman,   An Introduction to Homological Algebra,
Academic Press, New York, 1979.

\bibitem{SSW}S. Sather-Wagstaff,  {Semidualizing Modules,} \url{https://www.ndsu.edu/pubweb/~ssatherw/DOCS/survey.pdf}.

\bibitem{SSW08}S. Sather-Wagstaff, T. Sharif, and D. White,  Stability of Gorenstein categories, J. London Math. Soc. {\bf 77}(2008), 481--502.

\bibitem{TW10}R. Takahashi and D. White,  Homological aspects of semidualizing modules,  Math. Scand. {\bf 106}(2010), 5--22.

\bibitem{Va74}W. V. Vasconcelos,  Divisor Theory in Module Categories, North-Holland Mathematics
Studies, Vol. 14, North-Holland Publ. Co., Amsterdam, 1974.
\end{thebibliography}
\end{document}